\documentclass[11pt,a4paper,usenames,dvipsnames,reqno]{amsart}
\usepackage{amssymb}
\usepackage{amscd}
\usepackage{amsfonts}
\usepackage[margin=1in]{geometry}

\usepackage[all,arc]{xy}
\usepackage{enumerate}
\usepackage{mathrsfs}
\usepackage{color}   
\usepackage{hyperref}
\hypersetup{
    colorlinks=true, 
    linktoc=all,     
    linkcolor=blue,  
}
\usepackage{amsthm}
\usepackage{amsmath}
\usepackage{graphicx}
\usepackage{wasysym}
\usepackage{pgf,tikz}
\usetikzlibrary{positioning}
\usepackage{ marvosym }
\usepackage{tikz-cd}
\usepackage{ textcomp }
\usepackage{bbm}
\usepackage{ stmaryrd }

\newcommand{\cc}{\mathbb C}

\newcommand{\zz}{\mathbb Z}

\newcommand{\qq}{\mathbb Q}

\newcommand{\A}{\mathbb A}

\newcommand{\la}{\langle}
\newcommand{\ra}{\rangle}
\newcommand{\lra}{\longrightarrow}
\newcommand{\hra}{\hookrightarrow}

\newcommand{\al}{\alpha}
\newcommand{\be}{\beta}
\newcommand{\ga}{\gamma}
\newcommand{\de}{\delta}
\newcommand{\De}{\Delta}
\newcommand{\Ga}{\Gamma}
\newcommand{\ep}{\epsilon}
\newcommand{\lam}{\lambda}
\newcommand{\Lam}{\Lambda}
\newcommand{\vp}{\varpi}
\newcommand{\sig}{\sigma}
\newcommand{\ka}{\kappa}

\DeclareMathOperator{\Gal}{Gal}

\DeclareMathOperator{\End}{End}

\DeclareMathOperator{\G}{G}
\DeclareMathOperator{\GL}{GL}

\DeclareMathOperator{\Hom}{Hom}

\DeclareMathOperator{\U}{U}
\DeclareMathOperator{\rH}{H}

\DeclareMathOperator{\SO}{SO}

\DeclareMathOperator{\RO}{RO}

\DeclareMathOperator{\Sp}{Sp}

\DeclareMathOperator{\diag}{diag}

\DeclareMathOperator{\Spec}{Spec}

\DeclareMathOperator{\Ad}{Ad}

\DeclareMathOperator{\Res}{Res}
\DeclareMathOperator{\Orb}{O}
\DeclareMathOperator{\val}{val}
\DeclareMathOperator{\inv}{inv}

\newcommand{\fu}{\mathfrak u}

\newcommand{\fc}{\mathfrak c}

\newcommand{\fp}{\mathfrak p}

\newcommand{\fgl}{\mathfrak gl}

\newcommand{\cala}{\mathcal{A}}

\newcommand{\cals}{\mathcal{S}}
\newcommand{\calq}{\mathcal{Q}}
\newcommand{\calo}{\mathcal{O}}

\newcommand{\calv}{\mathcal{V}}
\newcommand{\cald}{\mathcal{D}}

\newcommand{\Herm}{\mathcal{H}erm}

\newcommand{\Nm}{\mathrm{Nm}}

\newcommand{\iso}{\xrightarrow{\sim}}
\newcommand{\car}{\textbf{car}}
\newcommand{\kbar}{{k^{alg}}}
\newcommand{\Fbar}{{F^{alg}}}
\newcommand{\bfun}{\mathbf{1}}

\makeatletter
\def\Ddots{\mathinner{\mkern1mu\raise\p@
\vbox{\kern7\p@\hbox{.}}\mkern2mu
\raise4\p@\hbox{.}\mkern2mu\raise7\p@\hbox{.}\mkern1mu}}
\makeatother

\newtheorem{Thm}{Theorem}[section]
\newtheorem{Prop}[Thm]{Proposition}

\newtheorem{Lem}[Thm]{Lemma}
\newtheorem{Cor}[Thm]{Corollary}

\theoremstyle{definition}
\newtheorem{Def}[Thm]{Definition}

\theoremstyle{remark}
\newtheorem{Rem}[Thm]{Remark}

\theoremstyle{definition}

\title[Descent and the fundamental lemma]{On the stabilization of relative trace formulae: descent and the fundamental lemma}
\author{Spencer Leslie}

\date\today

\address{Department of Mathematics, Duke University, 120 Science Drive, Durham, NC, USA}

\email{lesliew@math.duke.edu}


\begin{document}

\begin{abstract}
Motivated by the study of periods of automorphic forms and relative trace formulae, we develop the theory of descent necessary to study orbital integrals arising in the fundamental lemma for a general class of symmetric spaces over a $p$-adic field $F$. More precisely, we prove that a connected symmetric space over $F$ enjoys a notion of topological Jordan decomposition, which may be of independent interest, and establish a relative version of a lemma of Kazhdan that played a crucial role in the proof of the Langlands--Shelstad fundamental lemma.

As our main application, we use these results to prove the endoscopic fundamental lemma for the unit element of the Hecke algebra for the symmetric space associated to unitary Friedberg--Jacquet periods.

\end{abstract}

\maketitle
\setcounter{tocdepth}{1}
\tableofcontents


\section{Introduction}

The study of relative trace formulae (RTFs) is a major component of the (global) relative Langlands program of Jacquet-Sakellaridis-Venkatesh. The recent spectacular success of the proof of the (tempered) Gan-Gross-Prasad (GGP) conjecture for unitary groups \cite{ZhangFourier,beuzart2019isolation,beuzart2020global} via the Jacquet-Rallis trace formula comparison illustrates the power of this method. However, many of the good properties present in the GGP setting (trivial generic stabilizers on the geometric side, local multiplicity one on the spectral side) fail for more general RTFs, and several new problems arise. More generally, many RTFs require a form of \emph{endoscopy and stabilization} analogous to the setting of the Arthur-Selberg trace formula in order to be effectively compared. Though there are antecedents (\cite{GetzWambach}, for example), these issues have previously not been dealt with in any single example. As a first case, we recently proposed a theory of relative endoscopy for the case of unitary Friedberg--Jacquet periods \cite{Leslieendoscopy}, described in more detail below, and proved the infinitesimal version of the endoscopic fundamental lemma in \cite{LeslieUFJFL}. 

The goal of the present paper is to study relative $\ka$-orbital integrals that arise in the geometric expansion of RTFs associated to a general class of symmetric subgroups of connected reductive groups. We develop the tools necessary to reduce the associated fundamental lemma to the tangent space at the identity (the ``Lie algebra case''). Our approach is inspired by that of Waldspurger in twisted endoscopy, but new problems arise since we work on a variety and not a group. We establish several new geometric results for symmetric spaces over a $p$-adic field which, to the best of the author's knowledge, have no precedent in the literature. More specifically, we first show that there is a good theory of \emph{topological Jordan decomposition} for symmetric spaces over a non-archimedean local field of odd residue characteristic. We then prove a relative analogue of a lemma of Kazhdan, a generalization of which played a crucial role in the proof of the Langlands--Shelstad fundamental lemma, describing integral models of stabilizer group schemes of certain elements in the symmetric space.

Our main application is the case of unitary Friedberg--Jacquet periods. Extending our work on the Lie algebra, we develop a relative theory of endoscopy for symmetric spaces of the form $\U(2n)/\U(n)\times \U(n)$ and apply our general results to prove the endoscopic fundamental lemma for the unit element of the spherical Hecke algebra. This identity is the first such result in the literature and plays a central role in the stabilization of the elliptic part of the relative trace formula associated to the relevant automorphic period integrals \cite{LeslieElliptic}. 

For the rest of the introduction, we describe the motivation and content of our main result for unitary Friedberg--Jacquet periods (Theorem \ref{Thm: fundamental lemma intro} below), then turn to our general descent results as we outline the proof.
\subsection{Unitary Friedberg--Jacquet periods} 
Let $E/F$ be a quadratic extension of number fields, with $\A_E$ and $\A_F$ the associated rings of adeles. Let $W_1$ and $W_2$ be two $n$ dimensional Hermitian spaces over $E$. The direct sum $W_1\oplus W_2$ is also a Hermitian space and we have the embedding of unitary groups $$\U(W_1)\times \U(W_2)\hra \U(W_1\oplus W_2),$$ realizing $\U(W_1)\times \U(W_2)$ as the fixed-points of an involution.
Let $\pi$ be an irreducible cuspidal automorphic representation of $\U(W_1\oplus W_2)(\A_F)$. Then $\pi$ is said to  be \emph{distinguished} by the subgroup $\U(W_1)\times \U(W_2)$ if the \emph{period integral}
\begin{equation*}
\displaystyle\int_{[\U(W_1)\times \U(W_2)]}\varphi(h)dh
\end{equation*}
is not equal to zero for some vector $\varphi$ in the $\pi$-isotypic subspace of automorphic forms on $\U(W_1\oplus W_2)(\A_F)$. Here, $[\rH]=\rH(F)\backslash \rH(\A_F)$ for any $F$-group $\rH$. We call these \emph{unitary Friedberg-Jacquet periods} in homage to \cite{FriedbergJacquet}.

These periods have recently appeared in the literature in several ways (for example, \cite{IchinoPrasanna}, \cite{PollackWanZydor}, \cite{GrahamShah}, and indirectly in \cite{LiZhang}). We would therefore like to study automorphic forms distinguished by these subgroups. In 2018, Wei Zhang conjectured a comparison of RTFs relating these periods to special $L$-values of certain $L$-functions (see \cite{LeslieUFJFL} for a discussion about some related conjectures). Such a comparison is related to the general framework of Getz-Wambach \cite{GetzWambach}; indeed, a similar conjecture for these period integrals first appears in the thesis of J. Pol\'{a}k \cite{polakthesis}. We therefore study the RTF associated to unitary Friedberg--Jacquet periods on $\U(W_1\oplus W_2)(\A_F)$.

However, as indicated previously, this relative trace formula is not \emph{stable}. Spectrally, this should be related to the non-factorizability of the period integrals for certain automorphic representations. Geometrically, this instability manifests in that when we consider the action of $\U(W_1)\times \U(W_2)$ on the {symmetric space} 
\[
\calq:=\U(W_1\oplus W_2)/\U(W_1)\times \U(W_2),
\]
invariant polynomials distinguish only \emph{geometric (or stable) orbits}. Our goal is to stabilize the geometric side of this relative trace formula by considering relative analogues of the theory of endoscopy in order to establish these conjectures (as well as their generalizations).

In \cite{Leslieendoscopy}, we introduced a putative theory of relative endoscopy for the infinitesimal symmetric space (that is, the tangent space at the distinguished $U(W_1)\times U(W_2)$-fixed point of $\calq(F)$, which we refer to as the Lie algebra of $\calq$), and proved the existence of smooth transfer for many test functions. We then established the {fundamental lemma for the unit function} in this infinitesimal setting in \cite{LeslieUFJFL} via a combination of local harmonic analysis and a global argument relying on a new comparison of relative trace formulae. Considering the linearized case first was largely motivated by the case of (twisted) endoscopy, where both the existence of smooth transfer and the fundamental lemma for the unit function were ultimately reduced to the Lie algebra (\cite{Waldspurgerlocal},\cite{Waldstransfert},\cite{Waldspurgertordue} \cite{halesfundamental}, \cite{WaldsCharacteristic}, \cite{NgoFL}). We expect a similar reduction for many symmetric spaces.


This paper accomplishes the first step in this reduction by extending the relative theory of endoscopy to the symmetric space $\calq(F)$ and proving the fundamental lemma for the unit element by descending to the main result of \cite{LeslieUFJFL}. Section \ref{Section: descent intro} below describes our descent results needed to carry out this argument. We also establish the existence of smooth transfers for many test functions (see Proposition \ref{Prop: regular transfer}).

In forthcoming work, we use these results to stabilize the elliptic part of the relative trace formula associated to unitary Friedberg--Jacquet periods. We wish to emphasize however that the main result here is the core identity showing that our proposed theory of relative endoscopy gives rise to a global stabilization of the relative trace formula. 

\subsection{Main Result}Let us now state the main result. For brevity, we refer the reader to Section \ref{Section: rel end symmetric space} for the relevant notations for orbital integrals and transfer factors. 

Now assume $E/F$ is an unramified quadratic extension of $p$-adic fields with odd residue characteristic large enough depending only on $F/\qq_p$ (for example, if $F=\qq_p$ the only constraint is $p\neq2$).\footnote{Lemma \ref{Lem: eignvalue restrictions} is the only source of restriction, and is standard. We remark that when ${F}$ is a number field, this condition is satisfied at all but a finite number of places which may be explicitly calculated.} Let $\calo_F$ (resp. $\calo_E$) denote the ring of integers in $F$ (resp. $E$). We consider the symmetric pair
\[
(\G_n,\rH_n)=(\U(V_n\oplus V_n),\U(V_n)\times \U(V_n)),
\]
where $V_n$ is a \emph{split} Hermitian space of dimension $n$. Concretely, we set $V_n=E^n$ and equip it with the Hermitian form represented by the identity matrix $I_n$. This affords the self-dual lattice $\Lam_n=\calo_E^n\subset V_n$, with respect to which we obtain hyperspecial subgroups 
\[
\rH_n(\calo_F):=\U(\Lam_n)\times U(\Lam_n)\subset \rH_n(F)
\]and 
\[
\G_n(\calo_F):=\U(\Lam_n\oplus \Lam_n)\subset \G_n(F).
\]
Set $\calq_n:=\G_n/\rH_n$ and let $\bfun_{\calq_n(\calo_F)}$ to be the associated characteristic function. 

In Section \ref{Section: rel end symmetric space}, we introduce the notion of an unramified elliptic relative endoscopic datum $\Xi_{a,b}=(\xi_{a,b},\al,\be)$ where $a+b=n$, following the approach of \cite{Leslieendoscopy}; here $\xi_{a,b}$ is a certain unramified elliptic endoscopic datum of $\G_n$, $\al$ (resp. $\be$) is a Hermitian form on $E^a$ (resp. $E^b$). This gives rise to (a pure inner form of) an elliptic endoscopic groups $\G_{a,\al}\times \G_{b,\be}$ of $\G_n$ and a symmetric subgroup $\rH_{a,\al}\times \rH_{b,\be}$ of $\G_{a,\al}\times \G_{b,\be}$. Let
\[
\calq_{a,\al}\times \calq_{b,\be}:=\G_{a,\al}/\rH_{a,\al}\times \G_{b,\be}/\rH_{b,\be}
\]
be the associated endoscopic symmetric space.

For a regular semi-simple element $x\in \calq^{rss}_n(F)$, the endoscopic datum determines a character $\ka: \mathfrak{C}(\rH_{n,x},\rH_n;F)\to \cc^\times$ on a certain subgroup of the Galois cohomology of the stabilizer $\rH_{n,x}$ of $x$; see Section \ref{Section: Prelim inv} for a review of the necessary invariant theory. The set $\calo_{st}(x)$ of rational orbits inside the stable orbit of $x$ is naturally a torsor for this group, so we define the relative $\ka$-orbital integral
\[
\Orb^\ka(x,f)=\sum_{[x']\in \calo_{st}(x)}\ka(\inv(x,x'))\Orb(x',f),
\]
where $\inv(x,x')$ is the cohomological invariant associated to the rational orbit of $x'$ (see Section \ref{Section: Prelim inv}). When $\ka=1$ is the trivial character, set $\SO=\Orb^1$. 

We show in Section \ref{Section: rel end symmetric space} that there is a good notion of the matching of regular semi-simple elements
\[
x\in\calq_n^{rss}(F)\:\:\text{and}\:\:(x_a,x_b)\in\left(\calq_{a,\al}(F)\times \calq_{b,\be}(F)\right)^{rss},
\]
and transfer factors 
\[
\De_{rel}:\left(\calq_{a,\al}(F)\times \calq_{b,\be}(F)\right)^{rss}\times \calq_n^{rss}(F)\lra \cc
\]
in the sense that we can define the notion of smooth transfer of $\ka$-orbital integrals on $\calq_n(F)$ and stable orbital integrals on $\calq_{a,\al}(F)\times \calq_{b,\be}(F)$ (Definition \ref{Def: transfer variety}) and prove the existence of smooth transfers for many test functions (Proposition \ref{Prop: regular transfer}). 

 We now state our main result.

\begin{Thm}\label{Thm: fundamental lemma intro}
 If $(\al,\be) = (I_a,I_b)$, the functions $\bfun_{\calq_n(\calo_F)}$ and $\bfun_{\calq_a(\calo_F)}\otimes\bfun_{\calq_b(\calo_F)}$ match. Otherwise, $\bfun_{\calq_n(\calo_F)}$ matches $0$. 
 
 More precisely, for any regular semi-simple $x\in \calq_n(F)$ and matching elements $(x_a,x_b)
 \in \calq_{a,\al}(F)\times \calq_{b,\be}(F)$, if $\ka$ is the character associated to the endoscopic datum, then
\begin{equation}\label{first identity}
    \De_{rel}((x_a,x_b),x)\Orb^\ka(x,\bfun_{\calq_n(\calo_F)})=\begin{cases}\SO((x_a,x_b),\bfun_{\calq_a(\calo_F)}\otimes \bfun_{\calq_b(\calo_F)})&:(\al,\be)= (I_a,I_b),\\\qquad\qquad\qquad 0&:(\al,\be)\neq (I_a,I_b).\end{cases}
\end{equation}
\end{Thm}

\subsection{Descent of orbital integrals}\label{Section: descent intro}
As previously stated, the proof of Theorem \ref{Thm: fundamental lemma intro} uses methods of descent inspired by those in the case of twisted endoscopy \cite{Waldspurgertordue}. However, relative orbital integrals on a symmetric space are in general not as well understood as those on a group, and we must establish several new results in order to make this precise. 

In addition to the case of unitary Friedberg--Jacquet periods, there are several classes of symmetric space for which stabilization of the RTF is desirable for global spectral consequences. For example, the framework of \emph{twisted base change} as outlined in \cite{GetzWambach} anticipates global spectral results from such stabilizations. Examples include Galois symmetric pairs $(\Res_{E/F}G,G)$ with $G$ a reductive group over $F$, which are the subject of much current research (see \cite{prasad2015arelative,BPgalois,Flickerstab}); for a case of a different flavor, one may consider the non-tempered symmetric pair $(\U(2n),\Sp(2n))$ \cite{MitraOffenSp}. For these reasons, we consider a general connected symmetric space $(\G,\rH)$ over a non-archimedean field of odd residue characteristic, where $\G$ is assumed to be unramified.

Setting $\calq=\G/\rH$, we consider the integral points $\calq(\calo_F):=\G(\calo_F)/\rH(\calo_F)$. The main object of interest for the study of endoscopic fundamental lemma for $\calq$ is the (relative) $\ka$-orbital integral (see Section \ref{Section: Prelim inv} for notation)
\begin{equation}\label{eqn: orbital goal}
  \Orb^\ka(x,\bfun_{\calq_n(\calo_F)})=\sum_{[x']\in \calo_{st}(x)}e(\rH_{x'})\ka(\inv(x,x'))\Orb(x',f), 
\end{equation}
where $x\in \calq(\calo_F)$ is regular semi-simple, $e(\rH_{x'})=\pm1$ is the Kottwitz sign \cite{Kottwitzsign}, and  $\ka:\mathfrak{C}(\rH_x,\rH;F)\lra \cc^\times$. If we assume that $F=k((t))$ is a non-archimedean local field of positive characteristic for the moment, general arguments relate these values to $k$-point counts on certain quotient stacks (see \cite[Section 3]{YunSpringer} for example). The combinatorics of such a point count are prohibitive, and so one seeks to linearize the problem by showing that that (\ref{eqn: orbital goal}) may be expressed in terms of a linearization of the relevant stacks.

In this direction, for a non-archimedean local field $F$ of odd residue characteristic we develop the topological Jordan decomposition for certain elements of $\calq(F)$. More precisely, for elements $x\in \calq_n(F)$ which are \emph{strongly compact}, we prove in Proposition \ref{Prop: relative top decomp} that is a unique decomposition 
\[
x=x_{as}x_{tu}\:\text{ with }\: x_{as}, x_{tu}\in \calq(F)
\]
where $x_{as}$ is the \emph{absolutely semi-simple part} of $x$ and $x_{tu}$ is the \emph{topologically unipotent part} of $x$; see Section \ref{Section: top jordan decomp} for the relevant definitions. The point is that since $ x_{as}, x_{tu}\in \calq(F)$, one can hope to study relative orbital integrals inductively by passing to more regular elements. Given the applications of the classical topological Jordan decomposition to character formulae (\cite{adler2008good}, for example), we hope this result will have application to the study of relative characters for symmetric spaces.

Our proof of this decomposition relies on the existence of the symmetrization map 
\[s: \calq\lra \G,
\]
 which realizes the symmetric space as a closed subvariety of $\G$. This enables us to make sense of taking products of elements of $\calq(F),$ as well as the definitions of absolutely semi-simplicity and topological unipotency, which involve eigenvalue constraints. We then combine the algebraic properties of symmetric spaces with the structure of $p$-adic groups to obtain the decomposition. It is an interesting question whether an analogous theory exists for general spherical varieties of interest in the relative Langlands program.

We now assume that $F$ has characteristic $0$. With this decomposition, it now makes sense to discuss absolutely semi-simple elements of $\calq(F)$ and study their integrality properties. More precisely, for any semi-simple element $x\in \calq^{ss}(F)$, one considers the \emph{descendant} $\calq_x$ of $\calq$ at $x$. Defined in Section \ref{Section: Prelim inv}, this gives a lower-dimensional degeneration modelling $\calq$ near $x$. Suppose now that $(\G,\rH)$ admits a smooth $\calo_F$-model, and let $x_{as}\in \calq(\calo_F)$ be absolutely semi-simple. In Proposition \ref{Prop: relative Kazdhan lemma}, we show that the stabilizer group $\rH_{x_{as}}$ of $x_{as}$ admits smooth group $\calo_F$-scheme model, which imposes strong constraints on the stable orbits of such elements. This is a relative analogue of a result of Kottwitz \cite[Proposition 7.1]{Kottwitzstableelliptic} (itself a generalization of a lemma of Kazhdan \cite{Kazhdanlifting}). 

In particular, this implies that the descendant $\calq_{x_{as}}$ at an absolutely semi-simple point has a natural smooth integral model. Propositions \ref{Prop: relative top decomp} and \ref{Prop: relative Kazdhan lemma} combine to give a systematic reduction of the $\kappa$-orbital integrals arising in the fundamental lemma on $\calq(F)$ to orbital integrals on {descendants} at absolutely semi-simple elements. This is stated in Proposition \ref{Prop: absolute descent}, where orbital integrals at a regular semi-simple element $x=x_{as}x_{tu}$ are shown to equal orbital integrals on $\calq_{x_{as}}(F)$ at $x_{tu}$. Our proof relies on a fixed-point result of Edixhoven on finite-group actions on smooth schemes over $\calo_F$ \cite{edixhoven1992neron}; in particular, this step also relies on $\calq$ being a symmetric space (as opposed to just a smooth $\G$-scheme). Additional care is needed in our relative context as the generic stabilizer of regular semi-simple elements need not be connected nor abelian; for our result to hold, we impose a technical condition ensuring generic stabilizers are connected (see Definition \ref{Def: simplyconnected}).


The upshot is that the topological unipotent locus of $\calq(F)$ lies in the image of the exponential map. In particular, one may now pass to the tangent space at the $\rH(F)$-fixed point of $\calq(F)$ for a general class of symmetric spaces. We expect this result to play a key role in the stabilization of the (relatively) elliptic part of several relative trace formulae. The main obstacle is the correct notion of \emph{relative endoscopic data} and \emph{transfer factors}. The results of Section \ref{Section: relative endo integrals} give such a definition in the case $(\G_n,\rH_n)=(\U(V_n\oplus V_n),\U(V_n)\times \U(V_n))$; we prove the necessary descent of the transfer factor in Lemma \ref{eqn: almost there transfer factor}.

\subsection{The proof of Theorem \ref{Thm: fundamental lemma intro}}
Returning to the case of unitary Friedberg--Jacquet periods, we end the introduction with a short sketch of the proof of the main result. The idea is to use the results of Sections \ref{Section: TJD and descent} and \ref{Section: relative kazhdan} to descend to the fundamental lemma for the Lie algebra of $\calq_n$, which is the main result of \cite{LeslieUFJFL}. We recall this infinitesimal theory in Section \ref{Section: proof of FL}. 

Setting $W=V_n\oplus V_n$, we consider the Cayley transforms $\fc_\nu: \End(W) \dashrightarrow \GL(W)$, where $\nu=\pm1$. These exponential-like maps are well suited for the study of the symmetric space $\calq_n$. In particular, we introduce certain open subset $\calq_n^{\heartsuit,\nu}(F)\subset \calq_n(F)$, which we refer to as the $\nu$-very regular locus. We show in Section \ref{Section: very regular} that the Theorem \ref{Thm: fundamental lemma intro} may be readily reduced to the Lie algebra result via the Cayley transform whenever $x\in \calq_n^{\heartsuit,1}(F)\cup\calq_n^{\heartsuit,-1}(F)$. Combined with certain elementary vanishing properties of orbital integrals (Lemma \ref{Lem: reduce to hyperspecial}), it follows that (\ref{first identity}) is known unless $x\in \calq_n^{rss}(\calo_F)$ lies in the $\calo_F$-points of a certain singular sub-$\calo_F$-scheme; see Remark \ref{Rem: last cases}.

In Section \ref{Section: descent final}, we apply Proposition \ref{Prop: absolute descent} in these remaining degenerate cases and establish the necessary descent of the transfer factors. In fact, we use the structure of descendants in this case to establish a slight generalization (Lemma \ref{Lem: into the reg locus}) of this proposition to simplify the descent of transfer factors. The essential point is that the topologically unipotent part $x_{tu}$ lies in the very regular locus of the descendant at $x_{as}$, allowing us to pass to the Lie algebra. This concludes the final cases of the fundamental lemma.

\subsection{Outline}
In Section \ref{Section: Prelim}, we fix notation and review some background on invariant theory, symmetric pairs, and unitary groups. We recall the topological Jordan decomposition for unramified groups in Section \ref{Section: TJD and descent}, and develop the relative theory in Section \ref{Section: top jordan decomp rel}. In Section \ref{Section: relative kazhdan}, we introduce the notion of a nice, simply-connected symmetric pair, and  prove a relative version of Kazhdan's lemma in Proposition \ref{Prop: relative Kazdhan lemma}. Turning to orbital integrals in Section \ref{Section: orbital ints and reduction}, we prove the main descent identity in Section \ref{Section: absolute descent}. 

In Section \ref{Section: symmetric}, we specialize to case of unitary Friedberg--Jacquet periods and study the basic geometry of the symmetric space $\calq$. We also introduce the contraction map $R$ used in defining our transfer factors, and compute all semi-simple descendants of the symmetric space. In Section \ref{Section: relative endo integrals}, we define elliptic relative endoscopic data and the relevant symmetric spaces. We then define the matching of stable orbits and transfer factors, following the infinitesimal theory developed in \cite{Leslieendoscopy}, and state the main theorem as Theorem \ref{Thm: fundamental lemma}. These notions rely on the theory of endoscopy for unitary Lie algebras, which we review in Appendix \ref{Section: endoscopy roundup} for the convenience of the reader. Section \ref{Section: proof of FL} recalls the infinitesimal theory and fundamental lemma (stated as Theorem \ref{Thm: fundamental lemma Lie alg}). We then apply the Cayley transform in Section \ref{Section: cayley}, deducing the fundamental lemma over the very regular locus in Section \ref{Section: very regular}. We complete the proof of Theorem \ref{Thm: fundamental lemma intro} in Section \ref{Section: descent final}.

\subsection{Acknowledgements}
I want to thank Jayce Getz for suggesting studying relative notions of endoscopy and for many helpful suggestions while writing this paper. I am very grateful to Wei Zhang for many clarifying discussions and for his generosity of time and ideas. I also want to thank Yiannis Sakellaridis for several insightful conversations and for his interest in this work. Finally, I thank the anonymous referee for several helpful comments and corrections, which have clarified several parts of this paper. 

This work was partially supported by an AMS-Simons Travel Award and by NSF grant DMS-1902865.

\section{Preliminaries}\label{Section: Prelim}

\subsection{Invariant theory}\label{Section: Prelim inv}
For any field $F$ and any non-singular affine algebraic variety $\mathrm{Y}$ over $F$ with $\mathrm{G}$ a connected reductive algebraic group over $F$ acting algebraically on $\mathrm{Y}$, we set $\mathrm{Y}^{rss}$ to be the invariant-theoretic regular semi-simple locus. That is, $x\in \mathrm{Y}^{rss}$ if and only if its $\mathrm{G}$-orbit is of maximal possible dimension and Zariski-closed. We also recall the semi-simple locus $\mathrm{Y}^{ss}$ of points with Zariski-closed orbits. When $F$ is a local field of characteristic zero, and we endow $Y(F)$ with the Hausdorff topology, it is known \cite[Theorem 2.3.8]{AizGourdescent} that $x\in Y^{ss}(F)$ if and only if $\G(F)\cdot x\subset Y(F)$ is closed in the Hausdorff topology.

For $x,x'\in \mathrm{Y}^{rss}(F)$, we say that $x'$ is in the \emph{rational $\G(F)$-orbit} of $x$ if there exists $g\in \mathrm{G}(F)$ such that
\[
g\cdot x= x'.
\]
Fixing an algebraic closure $\Fbar$, two semi-simple points $x,x'\in \mathrm{Y}^{ss}(F)$ are said to lie in the same \emph{stable orbit} if $g\cdot x=x'$ for some $g\in \mathrm{G}(\Fbar)$ such that the cocycle
\[
\inv(x,x'):=[\tau\in \Gal(\Fbar/F)\mapsto \tau(g)^{-1}g]\in Z^1(F,\G_x)
\]
lies in $Z^1(F,\G^\circ_x)$, where $\G_x^\circ\subset \G_x$ is the connected component of the identity of the stabilizer of $x$ in $\G$. When the semi-simple stabilizers are all connected (see Lemma \ref{Lem: descendants}), this cocycle constraint is automatic.

A standard computation  shows that the set $\calo_{st}(x)$ of rational orbits in the stable orbit of $x$ are in natural bijection with
\[
\cald(\G^\circ_x,\G;F):=\ker\left[H^1(F,\G_x^\circ)\to H^1(F,\G)\right].
\]
Suppose that $F$ is non-archimedean and of characteristic zero. When $\G_x^\circ$ is a torus, $\cald(\G^\circ_x,\G;F)$ is a finite abelian group and $\calo_{st}(x)$ is naturally a $\cald(\G^\circ_x,\G;F)$-torsor. While this is true for our main application (Section \ref{Section: symmetric} and thereafter), many of our results apply to varieties with \emph{non-abelian} regular stabilizers. For these cases, the stabilization of the associated relative trace formula involves abelianized cohomology \cite{LabesseBook}. Following Labesse's formalism (see also \cite[Section III.2]{LabesseIntro}), we let 
\[
\mathfrak{C}(\G^\circ_x,\G;F):=\ker\left[H_{ab}^1(F,\G_x^\circ)\to H_{ab}^1(F,\G)\right].
\]
There is a natural injective map $\cald(\G^\circ_x,\G;F)\lra \mathfrak{C}(\G^\circ_x,\G;F)$ \cite[pg. 25]{LabesseIntro}, which is bijective if $\G_x^\circ$ is a torus. By composition, this gives a map
\begin{align}\label{eqn: abelianized}
   (\G_x^\circ\backslash\G)(F)&\lra \cald(\G^\circ_x,\G;F)\lra \mathfrak{C}(\G^\circ_x,\G;F),\\
x'&\longmapsto \inv(x,x').\nonumber
\end{align}

\subsection{Symmetric spaces}\label{Section: symmetric recap}
Let $F$ be a field. We recall standard notation and facts about symmetric spaces.
\begin{Def}
A \textbf{symmetric pair} is a triple $(\G,\rH,\theta)$ where $\rH\subset \G$ are reductive groups and $\theta$ is an involution of $\G$ such that $\rH$ is the fixed-point subgroup. The symmetric pair is \textbf{connected} if $\G/\rH$ is connected. The quotient variety $\calq:=\G/\rH$ is called the \textbf{symmetric space.}
\end{Def}
Fix a symmetric pair $(\G,\rH,\theta)$; we will frequently drop $\theta$ from the notation and write $(\G,\rH)$, which somewhat justifies the (standard) use of the word ``pair.'' Define the anti-automorphism of $\G$
\begin{equation*}
    \sig(g) := \theta\left(g^{-1}\right).
\end{equation*}
Denote 
\[
\G^\sig=\{g\in\G: \sig(g)=g\},
\]
and define the \emph{symmetrization map}
\begin{align*}
    s:\G&\lra \G^\sig\\
        g&\longmapsto g\sig(g).
\end{align*}
It is well known \cite[Lemma 2.4]{Richardson} that $s$ induces an embedding of affine $\G$-varieties
\[
s:\calq=\G/\rH\lra \G^\sig.
\]

Viewing $\calq$ as an $\rH$-variety, let $\calq^{ss}$ (resp. $\calq^{rss}$) denote the semi-simple locus (resp. regular semi-simple locus) of $\calq$. A torus $\mathrm{S}\subset\G$ is called \emph{$\theta$-split} if $\theta(s)=s^{-1}$ for all $s\in S$. It is well known that any two maximal $\theta$-tori are stably conjugate by an element of $\rH^\circ$ and that every semi-simple element of $\calq(\Fbar)$ is contained in a maximal $\theta$-split torus \cite[Theorem 7.5]{Richardson}. The rank of such a torus is called the \textbf{rank of the symmetric space} $\calq$ and is denoted $\mathrm{rank}(\calq)$.

The symmetrization map allows us to relate these notions to (regular) semi-simplicity of elements in $\G$.
\begin{Lem}\cite[Theorem 7.5]{Richardson}\label{Lem: semi-simple match}
 Using the symmetrization map, we identify $\calq\subset \G$ as a closed subvariety of $\G$. Then $x\in\calq$ is $\rH$-semi-simple if and only if $x\in \G^{ss}$ is semi-simple as an element of $\G$. In particular,
 \[
 \calq^{ss}=\calq\cap \G^{ss}.
 \]
\end{Lem}
The relationship between $\calq^{rss}$ and $\G^{rss}$ is more subtle. While in general they are unrelated, the symmetric spaces we consider in Section \ref{Section: symmetric} and beyond all satisfy
\begin{equation}\label{eqn: quasi-split condition}
    \calq^{rss}=\calq\cap \G^{rss};
\end{equation}
this is because the symmetric spaces we consider are \emph{quasi-split}: the centralizer in $\G$ of a maximal $\theta$-split torus $A\subset \calq$ is a torus; see \cite[Section 1.2]{LeslieSpringer} for a discussion on quasi-split symmetric pairs. It is easy to see that over an algebraically closed field (\ref{eqn: quasi-split condition}) is equivalent to $(\G,\rH)$ begin quasi-split.

Finally, for $x\in \calq^{ss}$, let $\G_x$ denote its centralizer as an element of $\G$ and $\rH_x$ its stabilizer. Then $(\G_x,\rH_x,\theta|_{\G_x})$ is a symmetric pair \cite[Definition 7.2.2]{AizGourdescent}, referred to as the \textbf{descendant} of $(\G,\rH)$ at $x$. We will also refer to the symmetric space $\calq_x:=\G_x/\rH_x$ as the descendant of $\calq$ at $x$. 

There are two natural closed immersions of $\calq_x$ into $\calq$. The first is simply given by restriction of $s$ to $\G_x$. For the second, define the \emph{symmetrization map at $x$} by
\begin{align*}
    s_x: \G_x&\lra \calq\\
        g&\longmapsto gx\sig(g)=xs(g)\nonumber.
\end{align*}
The following is immediate.
\begin{Lem}\label{Lem: symmetr at x}
 There is a commutative diagram
 \[
 \begin{tikzcd}
 &\G_x\ar[dl,swap,"s"]\ar[dr,"s_x"]&\\
 \calq_x\ar[rr,"x\cdot"]&&\calq.
 \end{tikzcd}
 \]
In particular, $s_x$ induces a closed immersion of affine $\G_x$-varieties
\begin{align*}
    s_x: \calq_x&\lra \calq\\
        y&\longmapsto xy\nonumber.
\end{align*}
\end{Lem}

\subsection{Local fields} For simplicity, we fix a non-archimedean local field $F$ of characteristic zero and assume that the residue characteristic $p$ is odd. A further assumption on $p$ will arise in Section \ref{Section: descent final}, but the results of the prior sections (aside from the main result, which relies on that section) are valid without this restriction. We note that the results of Sections \ref{Section: TJD and descent} and \ref{Section: relative kazhdan} apply more generally. Nevertheless, the characteristic zero assumption is needed for the main application.

We set $|\cdot|_F$ to be the normalized valuation so that if $\vp$ is a uniformizer, then
\[
|\vp|^{-1}_F= \#(\calo_F/\fp_F) = :q
\]
is the cardinality of the residue field $k:=\calo_F/\fp_F$. Here $\fp_F$ denotes the unique maximal ideal of $\calo_F$. 

Let $\Fbar$ denote a fixed algebraic closure of $F$ and $\calo_{\Fbar}\subset \Fbar$ its ring of integers. For $a\in \calo_{\Fbar}$, we let $\overline{a}\in \kbar$ denote its image in the residue field, and use similar notation for $k$.

For any quadratic \'{e}tale algebra $E/F$ of local fields, we set $\eta_{E/F}: F^\times \to \cc^\times$ for the character associated to the extension by local class field theory. We also let $\Nm_{E/F}:E^\times \lra F^\times$ denote the norm map.

Throughout the article, all tensor products are over $\cc$ unless otherwise indicated.

\subsection{Groups and Hermitian spaces}\label{Section: groups and stuff}
For a field $F$ and for $n\geq1$, we consider the algebraic group $\GL_n$ of invertible $n\times n$ matrices. Suppose that $E/F$ is a quadratic \'{e}tale algebra and consider the restriction of scalars $\Res_{E/F}(\GL_n)$. For any $F$-algebra $R$ and $g\in \Res_{E/F}(\GL_n)(R)$, we set \[
g\mapsto \overline{g}
\]
to be the Galois involution associated to the extension $E/F$. Set
\[
\mathrm{X}_n(F)=\{x\in \GL_n(E): x^\dagger:={}^t\overline{x}=x\},
\]
where $t$ denotes the transpose. Note that $\GL_n(E)$ acts on $\mathrm{X}_n(F)$ via
\[
g x=gxg^\dagger,\quad x\in \mathrm{X}_n(F),\: g\in \GL_n(E).
\]
 We let $\calv_n$ be a fixed set of orbit representatives. For any $x\in \mathrm{X}_n(F),$ set $\la\cdot,\cdot\ra_x$ to be the Hermitian form on $E^n$ associated to $x$. Denote by $V_x$ the associated Hermitian space and $\U(V_x)$ the corresponding unitary group. Note that if $g x=x'$ then 
\[
V_x\xrightarrow{{g}^\dagger}V_{x'}
\]
is an isomorphism of Hermitian spaces. Thus, $\calv_n$ gives a set of representatives $\{V_x: x\in \calv_n\}$ of the equivalence classes of Hermitian vector space of dimension $n$ over $E$. We will abuse notation and identify this set with $\calv_n$. If we are working with a fixed but arbitrary Hermitian space, we often drop the subscript. For any Hermitian space, we set
\[
U(V)=\U(V)(F).
\]
When $E/F$ is an unramified quadratic extension of $p$-adic fields, we fix $V_n=(E^n,I_n)$ as our representative of split Hermitian spaces.

\subsection{Measures and centralizers}\label{measures}
We now assume $F$ is a non-archimedean local field of characteristic zero. We will only consider integration with respect to unimodular groups $\G(F)$, so we fix a Haar measure $dg$ throughout. In general, when $\G$ is unramified and $\G(\calo_F)$ is a fixed hyperspecial maximal subgroup, we choose the canonical normalization of $dg$ giving $\G(\calo_F)$ volume $1$. Outside of this setting, we may fix an arbitrary Haar measure as the precise choices will not affect the results of this paper.

In Sections \ref{Section: relative endo integrals} and beyond, we work with unitary groups and tori therein, so pause we make a few conventions here. When $E/F$ is unramified, $V_n=(E^n,I_n)$ our split Hermitian space, and $\Lam_n=\calo_E^n\subset V_n$ is the standard self-dual lattice, we always fix the Haar measures giving the hyperspecial maximal subgroups $\GL(\Lam_n)\subset\GL(V_n)$ and $U(\Lam_n)\subset U(V_n)$ volume $1$.

We need also to consider the measures on regular semi-simple centralizers. Fix a Hermitian form $x$ and consider $U(V)=U(V_x)$. We will be interested in the \emph{twisted Lie algebra}
\[
\Herm(V)=\{\de\in \End(V): \la \de v, u\ra=\la v,\de u\ra\}.
\]
The group $U(V)$ acts on this space by the adjoint action, and an element $\de$ is regular semi-simple if its centralizer is a maximal torus $T_\de\subset U(V)$. There is a unique maximal subgroup of $T_\de$ and we choose the measure $dt$ on $T_\de$ giving this subgroup volume $1$. We will study orbital integrals over regular semi-simple orbits and always use the measures introduced here to define invariant measures on these orbits. 

\begin{Rem}
This convention fixes measures on various rational orbits in a given stable orbit compatibly in the sense of transfer of measures along an inner twisting (see \cite[Chapter 3]{RogawskiBook}).
\end{Rem}

\section{A topological Jordan decomposition for symmetric spaces}\label{Section: TJD and descent}
We begin by recalling the topological Jordan decomposition for elements of unramified reductive groups $\G$ over a $p$-adic field $F$. We impose the assumption that the residue characteristic is odd. We then develop a relative version of this for certain elements in $p$-adic symmetric spaces. 

\subsection{Topological Jordan decomposition}\label{Section: top jordan decomp}
 Let $\G$ be an unramified connected reductive algebraic group over $F$. We assume that $\G(\calo_F)$ is a hyperspecial maximal compact subgroup of $\G(F)$. We recall the notions of topologically unipotent and absolutely semi-simple elements as defined in \cite{Halesunramified}; see also \cite{SpiceTopological}.

For any profinite group $K$ with a normal pro-$p$-subgroup $L$ of finite index, the prime-to-$p$ part of the order of $K/L$ is independent of the choice of $L$; denote this integer by $c_K$. Now for our reductive $p$-adic group $\G(F)$, if we fix representatives of the finitely many conjugacy classes of maximal compact subgroups $K_1,\ldots, K_d$, we may set $c_{\G}$ to be the least common multiple of $c_{K_i}$.

\begin{Def}
We say that an element $\ga\in \G(F)$ is \textbf{topologically unipotent} if 
\[
\lim_{n\to \infty}\ga^{q^n}=1,
\]where $q=|k|$ is the size of the residue field of $F$.
\end{Def}
\begin{Def}
We call a semi-simple element $\ga\in \G(F)$ \textbf{absolutely semi-simple} if $\ga^{c_{\G}}=1$.
\end{Def}
\noindent
We say $\ga\in \G(F)$ is \emph{strongly compact} if satisfies the following equivalent criteria
\begin{enumerate}
    \item $\ga$ lies in a compact subgroup of $\G(F)$, and
    \item the eigenvalues of $\rho(\ga)$ are units in $\Fbar$ for some faithful finite-dimensional rational representation $\rho:\G(F)\to \GL(V)$ defined over $\Fbar$.
\end{enumerate}
Clearly, topologically unipotent and absolutely semi-simple elements are strongly compact.

For each strongly compact element $\ga\in\G(F)$, there exists a unique decomposition 
\[
\ga=\ga_{as}\ga_{tu}=\ga_{tu}\ga_{as},
\]where 
$\ga_{as}$ is absolutely semi-simple and $\ga_{tu}$ is topologically unipotent; this is known as the \emph{topological Jordan decomposition}. This may be constructed as follows \cite{Halesunramified}: let $l$ be a positive integer such that $q^l\equiv 1\pmod{c}$, and set $\ga_{as}=\lim_{m\to \infty}\ga^{q^{lm}}$ and $\ga_{tu}=\ga\ga_{as}^{-1}$.

\begin{Lem}
The product $\ga=\ga_{as}\ga_{tu}$ gives the topological Jordan decomposition of $\ga$.
\end{Lem}
We record the following useful fact about topological Jordan decompositions, referring the reader to \cite{SpiceTopological} for more information.
\begin{Lem}\cite[Lemma 2.25]{SpiceTopological}\label{Lem: decomp in center}
 Suppose that $\ga\in \G(F)$ is strongly compact element with topological Jordan decomposition $\ga=\ga_{as}\ga_{tu}$. Then $\ga_{as},\ga_{tu}\in Z(\G_\ga)(F)$.
\end{Lem}

\subsection{The case of symmetric spaces}\label{Section: top jordan decomp rel}
Suppose now that $(\G,\rH)$ is a connected symmetric pair over $F$. We have the embedding of algebraic varieties 
\begin{align*}
   \calq:=\G/\rH&\lra  \G\\
   g&\mapsto s(g) = g\sig(g),
\end{align*}
where $\sig(g)=\theta\left(g^{-1}\right)$. 
We show that the topological Jordan decompositions of strongly compact elements of $\calq(F)\subset\G(F)$ respects the inclusion into $\calq(F)$. Our arguments combine the algebraic properties of symmetric spaces with the structure of $p$-adic groups. We remark that it is an interesting question whether this structure may be defined intrinsically to $\calq(F);$ such a definition might illuminate possible generalizations to spherical varieties that are not symmetric.

\begin{Lem}\label{Lem: unipotent image}
Suppose that $x\in \G(F)$ is topologically unipotent such that $x\in \G^\sig(F)$. Set $V(x)\subset G$ for the Zariski closure of the cyclic subgroup of $\G$ generated by $x$. Then there exists $y\in V(x)(\Fbar)$ such that $\sig(y)=y$ and $x=y^2=s(y)$. In particular, $x\in \calq(F)$.
\end{Lem}

\begin{proof}
We first assume that $x=x_u$ is unipotent. In this case, the lemma reduces to the classical statement of Richardson \cite[Lemma 6.1]{Richardson}. 

We now assume that $x=x_s$ is semi-simple. If $x$ has finite order, then since it is topologically unipotent there exists a smallest $m\in \zz_{\geq1}$ such that $x^{q^m}=1$; in particular, the order of $x$ is a power of $p$. Thus there is a $k$ such that $V(x)\cong \mu_{p^k}$ is the $p^k$-th roots of unity. The squaring map is an automorphism of this group, proving the lemma in this case.

Suppose now that $x=x_s$ is semi-simple of infinite order so that $V(x)$ is a torus. Passing to a finite extension as necessary, we are free to assume that there exist a $F$-split maximal torus $T(F)\subset \G(F)$ containing $x$ which is maximally $\theta$-split \cite[Theorem 7.5]{Richardson}. Let $A\subset T$ denote the maximal $\theta$-split torus contained in $T$. Then $x\in A(F)$ and since $\theta$ acts via inversion on $A(F)$, it suffices to show that $x$ is a square in $A(F)$.

Considering the split subtorus $V(x)\subset A$, there is an isomorphism
\begin{align*}
  A(F)&\iso (F^\times)^n\\ 
    t&\mapsto (t_1,\ldots,t_n)
\end{align*} such that for $t\in V(x)(F)$ we have
\begin{align*}
    t&\mapsto (t_1,\ldots,t_k, 1,\ldots, 1).
\end{align*}
 Under this isomorphism, the involution $\theta=(\theta_1,\ldots,\theta_n)$ acts via inversion on each factor. If we set
 \[
 x\mapsto (x_1,\ldots,x_k,1,\ldots,1),
 \] the assumption that $x$ is topologically unipotent implies that $\lim_{m\to \infty}x_i^{q^m}=1$ for each $1\leq i\leq k$. 
 In particular, we have 
 \[
 x_i\in 1+\fp_F
 \]
 for each $i$. This subgroup of $\calo_F^\times$ is a finitely-generated $\zz_p$-module, hence is $2$-divisible. Selecting $y_i\in 1+\fp_F$ such that $y_i^2=x_i$ for each $1\leq i\leq k$, we obtain an element $y\in V(x)(F)$ such that $y^2=x$.  

Finally, let $x=x_sx_u$ denote the Jordan decomposition of $x$. As we have seen above, there are elements
\[
y_s\in V(x_s)(\Fbar)\text{  and  }y_u\in V(x_u)(\Fbar)
\]such that $y_s^2=x_s$ and $y_u^2=x_u$. Since
$V(x_s), V(x_u)\subset Z(\G_x)$ by Lemma \ref{Lem: decomp in center}, we see $y_sy_u=y_uy_s$, so that
\[
s(y_sy_u)=(y_sy_u)^2 =y_s^2y_u^2= x.
\]
To see that $y_sy_u\in V(x)(F^{alg})$, we claim that the product map
\begin{align*}
  V(x_s)\times V(x_u)&\lra V(x)\\
    (g,h)&\longmapsto g h
\end{align*} is an isomorphism. In particular, the result follows from knowing $y_s\in V(x_s)$ and $y_u\in V(x_u)$. To see the claim, first note that $x_s,x_u\in V(x)(F)$ \cite[Theorem 2.4.8]{Springerbook}. Therefore, 
\[
V(x_s),V(x_u)\subset V(x).
\]
Since $x_sx_u=x_ux_s$, the commutator map 
\begin{align*}
    [\cdot,\cdot]: \G\times \G&\lra \G\\
                    (g,h)&\longmapsto ghg^{-1}h^{-1},
\end{align*}
vanishes on a Zariski-dense subgroup of $V(x_s)\times V(x_u)$; it follows that $V(x_u)$ and $V(x_s)$ commute with one another in $V(x)$. Since $V(x_s)\cap V(x_u)=\{1\}$, the product map
\begin{align*}
  V(x_s)\times V(x_u)&\lra V(x)\\
    (g,h)&\longmapsto g h
\end{align*}
is an injective homomorphism, the image of which is a closed subgroup of $V(x)$. As $x=x_sx_u$ lies in this closed subgroup, the map is an isomorphism. 
\end{proof}
The next proposition is the main result of this section. It allows us to discuss absolutely semi-simple and topologically unipotent elements of a symmetric space.
\begin{Prop}\label{Prop: relative top decomp}
For any strongly compact element $x\in \G(F)$, let $x=x_{as}x_{tu}$ be the topological Jordan decomposition. Then $x\in \calq(F)$ if and only if $x_{as},x_{tu}\in \calq(F)$.
\end{Prop}
\begin{proof}
Note that if $\theta(x)= x^{-1}$, then
\[
\theta(x_{as})\theta(x_{tu}) = \theta(x) = x_{as}^{-1}x_{tu}^{-1}.
\]
Uniqueness of the topological Jordan decomposition then forces 
\[
\text{$\theta(x_{as}) = x_{as}^{-1}$ and $\theta(x_{tu})=x_{tu}^{-1}$.}
\]A similar argument works for the converse, showing that $x\in \G^\sig(F)$ if and only if $x_{as},x_{tu}\in \G^\sig(F)$.

Suppose first that there exists $v\in \G(\Fbar)$ such that $s(v) = x$. The previous lemma states that there is a $y_{tu}\in \G(\Fbar)$ such that $s(y_{tu}) = y_{tu}^2= x_{tu}$. We claim that $y_{tu}$ commutes with $x$. Indeed, since $\G_x$ is a Zariski-closed subgroup of $\G$ and $x_{tu}\in Z(\G_x)$ by Lemma \ref{Lem: decomp in center}, we see that $y_{tu}\in V(x_{tu})(\Fbar)\subset Z(\G_x)(\Fbar)$. Therefore,
\[
s(y_{tu}^{-1}v) = y_{tu}^{-1}s(v)y_{tu}^{-1}= y_{tu}^{-2}x=x_{as}.
\]
Conversely, if $s(y_{as}) = x_{as}$ and $y_{tu}$ is as in Lemma \ref{Lem: unipotent image}, then $y_{tu}x_{as}=x_{as}y_{tu}$ as they both lie in $Z(\G_x)(\Fbar)$ by Lemma \ref{Lem: decomp in center}. This implies
\[
s(y_{tu}y_{as}) = y_{tu}s(y_{as})y_{tu}= x_{as}y_{tu}^{2}=x.\qedhere
\]
\end{proof}

Suppose that $x\in \calq^{ss}(F)$ is strongly compact element with topological Jordan decomposition $x=x_{as}x_{tu}$. Note that since $x_{as}=\lim_{m}x^{q^{lm}}\in \G^{ss}(F)$, Lemma \ref{Lem: semi-simple match} and Proposition \ref{Prop: relative top decomp} imply that $x_{as}\in \calq^{ss}(F)$. Let $\calq_{x_{as}}$ denote the descendant at $x_{as}$. Recall the symmetrization map at $x_{as}$
\[
s_{x_{as}}: \calq_{x_{as}}\lra \calq.
\]
Since $x_{tu}\in \G_{x_{as}}(F)\cap \calq^{ss}(F)$ and remains topologically nilpotent viewed as an element of the centralizer, an application of Lemma \ref{Lem: unipotent image} to the symmetric pair $(\G_{x_{as}}, \rH_{x_{as}})$ implies that $x_{tu}\in \calq^{ss}_{x_{as}}(F)$. It is easy to see that $s_{x_{as}}(x_{tu})=x$.

\begin{Lem}\label{Lem: top nil reg rel}
Suppose that $(\G,\rH)$ is a connected symmetric pair. Suppose that $x=x_{as}x_{tu}\in \calq^{rss}(F)$ is strongly compact element. Then $x_{tu}\in \calq_{x_{as}}^{rss}(F)$, where $\calq_{x_{as}}$ is the descent of $\calq$ at $x_{as}$.
\end{Lem}
\begin{proof}
Suppose that $y\in\calq^{ss}(F)$ and $z\in \calq_y^{ss}(F)$, and consider the element $s_y(z)=yz\in \calq(F)$. We first claim that $s_y(z)\in \calq_y(F)$; that is, we claim that there exists $g\in \G_y(\Fbar)$ such that $s(g)= s_y(z)$. 

Since we have assumed that $z\in \calq^{ss}_y(F)$, it follows that there exists $h\in \G_y(\Fbar)$ such that $s(h)=z$. Thus there exists $g\in \G_y(\Fbar)$ such that
\[
s(g)=s_y(z) = yz = hy\theta(h)^{-1}\quad\text{ if and only if }\quad s(h^{-1}g)=y,
\]
so it suffices to prove the claim when $z=1$ and $s_y(1)= y$. But the assumption that $y\in \calq^{ss}(F)$ implies that $y$ lies in a maximal $\theta$-split torus $A\subset \calq\subset \G$ \cite[Theorem 7.5]{Richardson}. Picking any maximal torus $S\supset A$, we have $S\subset \G_y$ and (over the algebraic closure) the symmetrization map surjects $S$ onto $A$, proving that there exists $g\in \G_y(\Fbar)$ such that $s(g) = y$.

Now viewing $s_y(z)\in \calq_y(F),$ we have the equality
\[
 (\rH_y)_{s_y(z)}=(\rH_y)_z
 \]
 as the two elements differ by an element of the center of $\G_y$. Noting that $\rH_z=\rH\cap\G_z$ for any $z\in \calq(F)$, we see that 
for any $z\in \calq_{y}^{ss}(F)$
\[
(\rH_{y})_{z}= \rH_{y}\cap (\G_{y})_{z}=\rH\cap(\G_y\cap \G_z)=\rH_y\cap \rH_{z}.
\]
In particular, $(\rH_y)_{s_y(z)} \subset \rH_{s_y(z)}$,
from which we conclude that
\[
\dim((\rH_y)_{z})=\dim((\rH_y)_{s_y(z)})\leq\dim(\rH_{s_y(z)})
\]
for any $z\in \calq^{ss}_y(F)$.

In general, it follows from \cite[Theorem 9.11]{Richardson} that there exists an integer $m$ such that for all $y\in \calq^{ss}(F)$ 
\[
\dim(\rH_y)\geq m,
\] and that $y$ is regular if and only if this is an equality. Indeed, it follows from \emph{loc. cit.} that
\[
m=\dim(\mathrm{St}_{\rH}(A))
\]
 is the dimension of the stabilizer in $\rH$ of any maximal $\theta$-split torus $A$ in $\calq$. Let $y\in \calq^{ss}(F)$, and let $m_{y}$ denote the corresponding dimension for the descendant $\calq_y$. We claim that $$m_y\geq m.$$
To see that this proves the lemma, note that if  $z\in \calq_y^{ss}(F)$ has the property that $s_y(z)\in \calq^{rss}(F)$ is regular semi-simple, then 
\[
m\leq m_y\leq \dim((\rH_y)_{z})\leq\dim(\rH_{s_y(z)})=m
\] forcing $m=m_y$ and $z\in \calq_y^{rss}(F)$. Applying this to $y=x_{as}$ and $z=x_{tu}$ proves the lemma.

We now prove that $m_y\geq m$. As the statement is geometric, we are free to pass the algebraic closure, and thus assume that $F=\Fbar$ for the remainder of this proof. For $y\in \calq^{ss}(F)$ and $z\in \calq_y^{ss}(F)$ as before, there exists a maximally $\theta$-split torus $A\subset \calq\subset \G$ such that $y,z\in A(F)$. Indeed, since $z\in \calq^{ss}_y(F)$ we may first choose a maximal $\theta$-split torus $A'\subset \calq_y\subset\G_y$ such that $z\in A'(F)$. Since $y$ is central in $\calq_y(F)\subset\G_y(F)$, it is contained in all such tori, so that $y\in A'(F)$. Now take any maximal $\theta$-split torus $A\subset \calq$ containing $A'$. 

It follows immediately from $y,z\in A(F)$ that $\mathrm{St}_{\rH}(A)\subset \rH_y\cap \rH_z=(\rH_{y})_{z}$, implying $$\dim((\rH_{y})_{z}) \geq \dim(\mathrm{St}_{\rH}(A))= m.$$ This proves the desired inequality $m_y\geq m$.
\end{proof}

We will use this in our descent of orbital integrals on $\calq(F)$ at $x$ to those on $\calq_{x_{as}}(F)$ at $x_{tn}$ in Section \ref{Section: absolute descent}; see also Lemma \ref{Lem: new decomp}.

\section{A relative Kazhdan's lemma and descent}\label{Section: relative kazhdan}
In this section, we develop a version of an important lemma of Kazhdan (see \cite[pg. 1364]{Halesunipotent} for the statement and references) for symmetric spaces. This allows us to descend the orbital integrals to certain descendants of $\calq(F)$ arising from the relative topological Jordan decomposition. We prove these results for a general class of symmetric pairs, which we specify in the next subsection.


\subsection{Nice and simply-connected symmetric spaces}


We assume that the symmetric pair $(\G,\rH)$ arises from a symmetric pair over the ring of integers $\calo_F$ in the sense that there is a smooth reductive group scheme $\mathcal{G}$ over $\calo_F$ and an involutive automorphism 
\[
\theta: \mathcal{G}\lra \mathcal{G}
\]
such that $\theta:\G\lra \G$ arises as the generic fiber. Set $\mathcal{H}=\mathcal{G}^\theta$. This gives a smooth group scheme over $\calo_F$ \cite[Proposition 3.4]{edixhoven1992neron} with reductive neutral component such that $\mathcal{H}_F\cong \rH$. The subscript $F$ here denotes the base change from $\calo_F$ to $F$. In particular, $\G$ and $\rH$ are both unramified as groups over $F$.

Consider now the symmetrization map 
\begin{align*}
    s:\mathcal{G}&\lra \mathcal{G}\\
        g&\longmapsto g\theta(g)^{-1},
\end{align*}
and define $\mathcal{Q}$ to be the scheme-theoretic image of $s$. By definition, the induced morphism $s: \mathcal{G}\to \mathcal{Q}$ is dominant and $\calq\subset \mathcal{G}$ is a closed embedding of affine $\calo_F$-schemes of finite type. Note also that $\calq$ is finitely presented \cite[\href{https://stacks.math.columbia.edu/tag/00FP}{Tag 00FP}]{stacks-project}.  

The following lemma is surely well known; we could not locate a reference so include a proof for the completeness.
\begin{Lem}\label{Lem: integral quotient}
 The scheme $\calq$ is smooth over $\calo_F$ and represents the fppf quotient $\mathcal{G}/\mathcal{H}$.
\end{Lem}
\begin{proof}
 Since $\mathcal{G}$ and $\mathcal{H}$ are smooth affine group schemes of finite type over $\calo_F$, the fppf quotient $\mathcal{G}/\mathcal{H}$ is represented by an $\calo_F$-scheme, which we will also call $\mathcal{G}/\mathcal{H}$ \cite[Th\'{e}or\`{e}me 4.C]{AnaSiva}. Let $q:\mathcal{G}\to \mathcal{G}/\mathcal{H}$ denote the resulting faithfully flat morphism of $\calo_F$-schemes; it is an $\mathcal{H}$-torsor over $\calo_F$. In particular, $\mathcal{G}/\mathcal{H}$ is smooth and affine over $\calo_F$ since both $\mathcal{G}$ and $\mathcal{H}$ are \cite[Remark 6.5.2]{PoonenRational}. 
 
 The morphism $s$ clearly factors through $\mathcal{G}/\mathcal{H}$, so we obtain a morphism
 $\tilde{s}:\mathcal{G}/\mathcal{H}\lra \calq$ such that $s= \tilde{s}\circ q$. To prove the lemma, we must show $\tilde{s}$ is an isomorphism of $\calo_F$-schemes. 
 The key input is Lemma 2.4 of \cite{Richardson}, which implies that the base change \begin{equation}\label{eqn: check on fibers}
     \tilde{s}_{\Spec(K)}:(\mathcal{G}/\mathcal{H})_{\Spec(K)}\iso \calq_{\Spec(K)}
 \end{equation}
for any field $K$ over $\calo_F$. Note that both schemes are finitely presented over $\Spec(\calo_F)$ with $\mathcal{G}/\mathcal{H}$ smooth over $\Spec(\calo_F)$. The fibral isomorphism criterion \cite[Corollaire 17.9.5]{EGAIV} now implies $\tilde{s}$ is an isomorphism by applying \eqref{eqn: check on fibers} to the residue fields of points of $\Spec(\calo_F)$.
\end{proof}
In particular, the generic fiber of $\calq$ is the symmetric space $\G/\rH$, which by a slight abuse of notation we continue to call $\calq$. We have the canonical inclusion $\calq(\calo_F)\subset \calq(F)$.
\begin{Def}\label{Def: simplyconnected}
We define a symmetric pair $(\G,\rH)$ over a field $k$ to be \textbf{simply connected} if for every field extension $K/k$ and every semi-simple point $x\in \calq(K)$, the centralizer $(\rH_K)_x$ is connected. We say a symmetric pair $(\mathcal{G},\mathcal{H})$ over $\calo_F$ is simply connected if both the generic and special fibers are simply connected.
\end{Def}
Note that the above condition forces, in particular, $\rH$ to be connected. One of the reasons for restricting to simply-connected symmetric pairs is the following surjectivity statement, which fails for general symmetric pairs (e.g. for $(\GL_n,\mathrm{O}_n)$).
\begin{Cor}\label{Cor: surjects on integral points}
Assume that $(\mathcal{G},\mathcal{H})$ is a simply-connected symmetric pair over $\calo_F$, and let $\calq$ be the associated symmetric space over $\calo_F$. Then the map $s$ is surjective on $\calo_F$-points: we have a short exact sequence of pointed sets
\[
1\lra \mathcal{H}(\calo_F)\lra \mathcal{G}(\calo_F)\overset{s}{\lra} \calq(\calo_F)\lra 1.
\]
\end{Cor}
\begin{proof}
  Since the special fiber $\mathcal{H}_k$ is a smooth, connected reductive group scheme over $k$, a theorem of Lang \cite[Corollary to Theorem 1]{Lang} shows that $H^1(k,\mathcal{H}_k)=1$. Hensel's lemma and the smoothness of $\mathcal{H}$ now implies that $H^1_{\text{\'{e}t}}(\Spec(\calo_F), \mathcal{H})=1.$ That is, all $\mathcal{H}$-torsors over $\calo_F$ are trivial.
  
  Now for any $x\in \calq(\calo_F)$, the fiber $q^{-1}(x)\subset \mathcal{G}$ is an $\mathcal{H}$-torsor over $\calo_F$ since $\mathcal{G}\to \calq$ is. It must be trivial by the aforementioned vanishing of $H^1_{\text{\'{e}t}}(\Spec(\calo_F), \mathcal{H})$, implying $q^{-1}(x)(\calo_F)\neq \emptyset$.
\end{proof}
\begin{Def}
We say that a symmetric pair $(\mathcal{G},\mathcal{H})$ over $\calo_F$ is \textbf{nice} if the ring of invariants
\[
\calo_F(\calq)^\mathcal{H}
\]
is a finitely generated $\calo_F$-algebra such that for every $x\in \Spec(\calo_F)$
\[
\calo_F(\calq)^\mathcal{H}\otimes_{\calo_F}k_x\cong k_x(\calq_x)^{\mathcal{H}_x},
\]
where $x=\Spec(k_x)$, and the subscripts denote passing to the fiber at $x$.
\end{Def}
If this holds, the $\calo_F$-scheme $\cala:=\Spec(\calo_F(\calq)^\mathcal{H})$ has the property that for each $x\in \Spec\calo_F$, the fiber $\cala_x$ is the categorical quotient for the $\mathcal{H}_x$ action on $\calq_x$.

Our primary example is $G=\U(V_n\oplus V_n)$ and $H= \U(V_n)\times \U(V_n),$ where $V_n$ is a split Hermitian space of dimension $n$ for an unramified extension $E/F$. We consider this case in detail in the next section, proving it is nice and simply connected in Lemma \ref{Lem: unitary nice}. 
We suspect that smooth symmetric pairs over $\calo_F$ are always nice, but do not have a proof. It is relatively easy to check once a concrete model for the categorical quotient is constructed. Other examples of nice simply-connected pairs are Galois pairs associated to simply-connected groups such as the symmetric pair $(\Res_{E/F}\rH_E,\rH)$ for an unramified quadratic extension $E/F$, where $\rH_E$ denotes the base change.

\begin{Prop}\label{Prop: relative Kazdhan lemma}
Suppose that $(\mathcal{G},\mathcal{H})$ is a nice simply-connected symmetric pair over $\calo_F$ and suppose that $\ga\in \calq(\calo_F)$ is absolutely semi-simple. Let $(\G,\rH)$ denote the associated symmetric pair over $F$. The centralizer $\rH_\ga$ is unramified and arises as the generic fiber of a smooth connected reductive group scheme $\mathcal{H}_\ga\subset \mathcal{H}$ over $\calo_F$. 

Moreover, if $\ga$ and $\ga'\in \calq(\calo_F)$ lie in the same stable $\rH(F)$-orbit, they are conjugate by an element in $\mathcal{H}(\calo_F)$.
\end{Prop}
\begin{proof}
To begin, we explain that \cite[Proposition 7.1]{Kottwitzstableelliptic} applies to absolutely semi-simple elements of $\G(F)$. This is implicit in \cite{Halesunramified}, and we provide the argument for the convenience of the reader. Suppose that $\ga\in \mathcal{G}(\calo_F)$ is absolutely semi-simple and suppose $T\subset G$ is a maximal torus such that $\ga\in T(F)$. Our assumption on $\ga$ implies that $\ga^{c_G}=1$ for $c_G$ defined as in Section \ref{Section: TJD and descent}. In particular, for each root $\al$ of $(G,T)$, $\al(\ga)\in \mu_{c_G}(\Fbar)$. Since $(c_G,p)=1$ by definition, it follows that $1-\al(\ga)\in \calo_{\Fbar}$ is either $0$ or a unit for each root $\al$. This is precisely Kottwitz' criterion.

Now suppose that $\ga\in \mathcal{Q}(\calo_F)$ is absolutely semi-simple as an element of $\mathcal{G}(\calo_F)$. Then \cite[Proposition 7.1]{Kottwitzstableelliptic} implies that $\mathcal{G}_\ga$ is a smooth group scheme over $\calo_F$ with reductive fibers. It is evidently stable under $\theta$, so we consider the automorphism 
\[
\theta:\mathcal{G}_\ga\lra\mathcal{G}_\ga.
\]
The fixed point scheme is given by
\[
\mathcal{G}^\theta_\ga(R)=\{g\in \mathcal{G}_\ga(R):\theta(g) = g\} = \mathcal{H}_\ga(R),
\]
for any $\calo_F$-algebra $R$. It follows from \cite[Proposition 3.4]{edixhoven1992neron} that $\mathcal{H}_\ga$ is a smooth group scheme over $\calo_F$. By our assumption that $(\mathcal{G},\mathcal{H})$ is simply connected, the smooth group scheme $\mathcal{H}_\ga$ has connected reductive fibers. In particular, $\rH_\ga = \mathcal{H}_{\ga,F}$ is unramified and $\mathcal{H}_{\ga}(\calo_F)=H_\ga(F)\cap \mathcal{H}(\calo_F)$ is a hyperspecial maximal subgroup. 

Now suppose that $\ga$ and $\ga'\in \calq(\calo_F)$ are absolutely semi-simple and lie in the same stable orbit. Since the stabilizers are connected, this implies that $\ga'=h\cdot\ga$ for some $g\in \rH(\Fbar)$. Viewed as elements of $\calq(F)$, it follows that $\ga$ and $\ga'$ have the same invariant $a\in \cala(F)$, where $\cala:=\Spec(F(\calq)^H)$ denotes the categorical quotient. By the assumption of niceness, the quotient map
\[
\calq\lra\Spec(F(\calq)^H)
\]
has a natural $\calo_F$-model, which we also call $\cala$. This $\calo_F$-scheme satisfies the property that for each point $x\in \Spec(\calo_F)$, the fiber $\cala_x$ is the categorical quotient of $\calq_x$ with respect to $\rH_x$. We have the commutative diagram
\[
\begin{tikzcd}
\calq(\calo_F)\ar[r]\ar[d]&\calq(F)\ar[d]\\
\cala(\calo_F)\ar[r]& \cala(F),
\end{tikzcd}
\]where the horizontal arrows are the natural inclusions. In particular, $a\in \cala(\calo_F)$.

Define now the $\calo_F$-scheme given by
\[
Y(R) = \{g\in \mathcal{G}(R): g\ga g^{-1}=\ga'\}
\]
for any $\calo_F$-algebra $R$. By the proof of Proposition 7.1 of \cite{Kottwitzstableelliptic}, we know that $Y$ is smooth as an $\calo_F$-scheme and that $Y(\calo_F)\neq\emptyset$. It is simple to check that the involution $\theta$ preserves $Y$ and so another application of \cite[Proposition 3.4]{edixhoven1992neron} implies that $Y^\theta$ is a smooth scheme over $\calo_F$ such that for any $\calo_F$-algebra $R$
\[
Y^\theta(R) = \{g\in \mathcal{G}(R): g\ga g^{-1}=\ga',\: \theta(g) =g \}= \{g\in \mathcal{H}(R): g\ga g^{-1}=\ga'\}.
\]
To prove the final claim, it suffices to show that $Y^\theta(\calo_F)\neq \emptyset$.

Let $\overline{\ga}$ and $\overline{\ga}'$ denote the images of $\ga$ and $\ga'$ in $\calq(k)$. These elements are semi-simple as $\ga$ and $\ga'$ are absolutely semi-simple as elements of $\mathcal{G}(\calo_F)$. Since the quotient map%
\begin{align*}
    \cala(\calo_F)&\lra \cala(k)\\
    a&\longmapsto\overline{a}
\end{align*}
is functorial, $\overline{\ga}$ and $\overline{\ga}'$ have the same invariant $\overline{a}\in \cala(k)$. As there is a unique stable semi-simple orbit in the fiber over $\overline{a}\in \cala(k)$ (see \cite[Theorem 2.2.2]{AizGourdescent}), it follows that $\overline{\ga}$ and $\overline{\ga}'$ lie in the same stable orbit under the action of $\mathcal{H}(k)$. 

By the assumption that $(\mathcal{G},\mathcal{H})$ is a simply-connected symmetric pair, the stabilizer $\mathcal{H}_{\ga,k}$ is connected. Lang's theorem \cite[Corollary after Theorem 1]{Lang} now implies the vanishing of Galois cohomology $H^1(k,\mathcal{H}_{\ga,k})=0$. In particular, there is a single $\mathcal{H}(k)$-orbit in the stable orbit of $\overline{\ga}$ so that  $\overline{\ga}$ and $\overline{\ga}'$ lie in the same $\mathcal{H}(k)$-orbit. That is,
\[
Y^\theta(k)\neq\emptyset.
\]
The smoothness of $Y^\theta$ over $\calo_F$ and Hensel's lemma now gives that $Y^\theta(\calo_F)\neq \emptyset$.
\end{proof}

\subsection{Good orbits} Our main application of Proposition \ref{Prop: relative Kazdhan lemma} will come in Proposition \ref{Prop: absolute descent}. We derive a few more consequences here to be used in Sections \ref{Section: absolute descent} and \ref{Section: descent final}. Suppose $(\G,\rH)$ are as in the proposition; we continue to use calligraphic font of integral models and make use of the canonical inclusions $\mathcal{G}(\calo_F)\subset \G(F)$ and $\calq(\calo_F)\subset \calq(F)$.

\begin{Def}
Let $\mathfrak{O}\subset \G(F)$ be a closed $\rH(F)\times \rH(F)$-orbit. We say that $\mathfrak{O}$ is a \textbf{good} orbit if 
\[
\sig(\mathfrak{O})=\mathfrak{O},
\]
where we remind the reader that $\sig(g) = \theta(g)^{-1}$. We say that a closed $\rH(F)$-orbit in $\calq(F)$ is good if it is the image of a good $\rH(F)\times \rH(F)$-orbit under the symmetrization map.
\end{Def}
This terminology is inspired by the notion of a good symmetric pair from \cite[Section 7]{AizGourdescent}.

\begin{Cor}\label{Cor: absolute good orbit}
Suppose that $\ga\in \calq(\calo_F)$ is absolutely semi-simple. The orbit $\mathfrak{O}_\ga:=\rH(F)\cdot \ga$ is good.
\end{Cor}
\begin{proof}
First, let $x\in \calq(F)$ be any semi-simple element such that there exists $g\in \G(F)$ such that $s(g)=x$. To show that the orbit of $g$ is good, it suffices to show that $g$ lies in the same $\rH(F)\times\rH(F)$-orbit of $\sig(g)$. Lemma 7.1.4 of \cite{AizGourdescent} implies that $\sig(g)$ lies in the \emph{stable} $\rH\times \rH$-orbit of $g$.
In particular, $s(g)$ lies in the same stable $H$-orbit as 
\[
\tilde{s}(g) := s(\sig(g)) = \sig(g)g.
\]

Now let $\ga\in \calq(\calo_F)$ be absolutely semi-simple as in the statement. Then Corollary \ref{Cor: surjects on integral points} implies there exists $g\in \mathcal{G}(\calo_F)$ with $s(g) = \ga$ and we show that the $H(F)\times H(F)$-orbit of $g$ is good. Note that 
 \[
 \tilde{\ga} = \tilde{s}(g)= \sig(g)g
 \]
 also lies in $\calq(\calo_F)$. Proposition \ref{Prop: relative Kazdhan lemma} now implies that $\ga$ and $\tilde{\ga}$ lie in the same $\mathcal{H}(\calo_F)$-orbit. It is easy to see that this implies that
 \[
 \sig(g) = h_1gh_2
 \]
 for some $h_1,h_2\in H(F),$ proving the claim.
\end{proof}

\begin{Cor}\label{Cor: absolute nice lifts}
Suppose that $\ga\in \calq(\calo_F)$ is absolutely semi-simple. Then there exists $g\in \mathcal{G}_{\ga}(\calo_F)$ such that $s(g) = \ga$.
\end{Cor}
\begin{proof}
As before, Lang's theorem implies that there exists $g_0\in \mathcal{G}(\calo_F)$ such that $s(g_0)=\ga$. By Corollary \ref{Cor: absolute good orbit}, the $H(F)\times H(F)$-orbit of $g_0\in \mathcal{G}(\calo_F)$ is good. In particular, there exist $h_1,h_2\in H(F)$ such that $\sig(g_0) = h_1g_0h_2$. Set $g':=g_0h_1$ and note that $\ga=s(g')$. 

We claim that
\[
g'\sig(g') = \sig(g')g'.
\]
Indeed,
\begin{align*}
\sig(g')g'=h_1^{-1}\sig(g_0)g_0h_1&=h_1^{-1}\sig(g_0)\sig(\sig(g_0))h_1\\
                            &=h_1^{-1}(h_1g_0h_2)(h_2^{-1}\sig(g_0)h_1^{-1})h_1=g_0{\sig}(g_0)= g_0h_1(h_1^{-1}\sig(g_0))= g'\sig(g').    
\end{align*}
This now implies that 
\[
(g')^{-1}\ga g'= (g')^{-1}g'\sig(g') g'=\sig(g')g' = g\sig(g') = \ga,
\]
so $g'\in G_\ga(F)$. Inspecting the previous argument, we find the equation
\[
\sig(g_0)g_0=h_1g'\sig(g')h_1^{-1}=h_1\ga h_1^{-1}.
\]Noting that $\sig(g_0)g_0\in \calq(\calo_F)$, Proposition \ref{Prop: relative Kazdhan lemma} now implies that
\[
h_1\in \mathcal{H}(\calo_F)H_\ga(F).
\]
Write $h_1 = hh'$ for $h\in \mathcal{H}(\calo_F)$ and $h'\in H_\ga(F)$, and set $g:=g_0h\in \mathcal{G}(\calo_F)$. Then clearly $s(g)=\ga$ and
\[
g= g'(h')^{-1}\in \mathcal{G}_\ga(\calo_F).\qedhere
\]
\end{proof}
\subsection{Orbital integrals}\label{Section: orbital ints and reduction} Fix a connected symmetric pair $(\G,\rH)$ over $F$. We first define the local relative orbital integrals that come most directly from the relative trace formula associated to (global) $\rH$-periods, then reduce via the symmetrization map to orbital integrals on the symmetric space. First, we need to introduce the following terminology.
\begin{Def}
We say that $\ga\in \G(F)$ is \emph{relatively (resp. regular) semi-simple} if $s(\ga)\in \calq(F)$ is (resp. regular) semi-simple with respect to the $\rH(F)$-action.
\end{Def}

Set $\rH_1:=\rH$. Fix an element $x_2\in \calq(F)$ and let $\rH_2\subset \G$ denote its stabilizer in $\G$ under the twisted-conjugation action:
\begin{align*}
    \G\times \calq&\lra \calq\\
    (g,x)&\longmapsto gx\sig(g).
\end{align*}
Then $\rH_2$ is the fixed-point subgroup of $\G$ with respect to the involution 
\[
\theta_x(h) = x\theta(h)x^{-1}
\]and gives a pure inner form of $\rH_1$. Set $s_2:\G\lra \calq$ to be $s_2(g)=gx_2\sig(g)$.

For $f\in C_c^\infty(\G(F))$, and $\ga\in \G(F)$ a relatively semi-simple element, we define the {relative orbital integral} of $f$ by
\begin{equation*}
\RO(\ga,f) = \displaystyle\iint_{(\rH_1\times \rH_2)_\ga(F)\backslash \rH_1(F)\times \rH_2(F)}f(h_1^{-1}\ga h_2) {d\dot{h}_1d\dot{h}_2},
\end{equation*}
where $d\dot{h}_1d\dot{h}_2$ denotes the invariant measure determined by our choice of Haar measures on $\rH_i(F)$ and $(\rH_1\times \rH_2)_\ga(F)$. 

 We now explain the reduction to orbital integrals on the symmetric space for regular semi-simple orbits. 
 Assume that $\ga$ is regular relatively semi-simple, then \cite[Proposition 7.2.1]{AizGourdescent} implies that $(\rH_1\times \rH_2)_\ga \cong \rH_{s_2(\ga)}$. Set $x=s_2(\ga)\in\calq^{rss}(F)$. 
The pushforward map $(s_2)_!:C_c^\infty(\G(F))\lra C^\infty_c(\calq(F))$ given by
\[
(s_2)_!(f)(s_2(g)) = \displaystyle \int_{\rH_2(F)}f(gh)dh
\]
is surjective onto the sub-module $C^\infty_{c,2}(\calq(F))$ of functions whose support is contained in the image of $s_2$. Setting $\Phi:=(s_2)_!(f)$ and $x=s_2(\ga)$, the isomorphism $(\rH_1\times \rH_2)_\ga \cong \rH_{x}$ and absolute convergence of the relative orbital integral gives
\begin{align}\label{eqn: OI first reduction}
\Orb(x,\Phi):= \int_{\rH_x(F)\backslash \rH(F)}\Phi(h^{-1}xh){d\dot{h}}=\RO(\ga,f),
\end{align}
where $d\dot{h}$ denotes the invariant measure on $\rH_x(F)\backslash \rH(F)$ induced from our choice of Haar measures. More generally, for any $f\in C^\infty_{c}(\calq(F))$ and regular semi-simple $x\in \calq^{rss}(F),$ we set
\begin{align*}
\Orb(x,f)= \int_{\rH_x(F)\backslash \rH(F)}f(h^{-1}xh){d\dot{h}}
\end{align*}
to be the orbital integral of $f$ at $x$.

Finally, let $\ka: \mathfrak{C}(\rH_x,\rH;F)\to \cc^\times$ be a character. Recalling (\ref{eqn: abelianized}), we define the relative $\kappa$-orbital integral by
\begin{align*}
    \Orb^\ka(x,f)=\sum_{[x']\in \calo_{st}(x)}e(\rH_{x'})\ka(\inv(x,x'))\Orb(x',f),
\end{align*}
where $e(\rH_{x'})$ is the Kottwitz sign of $\rH_{x'}$ \cite{Kottwitzsign}. When $\ka=1$ is the trivial character, we use the standard notation $\SO(x,-)$ and refer to this as the \emph{stable orbital integral} at $x$.

\subsection{Descent of orbital integrals}\label{Section: absolute descent}
Now suppose that $(\mathcal{G},\mathcal{H})$ is a nice, simply-connected symmetric pair, and let $(\G,\rH)$ be the generic fiber. Let $x=x_{as}x_{tu}\in \calq^{rss}(\calo_F)$ be the topological Jordan decomposition. Proposition \ref{Prop: relative top decomp} implies that $x_{as},x_{tu}\in \calq(\calo_F)$.  Proposition \ref{Prop: relative Kazdhan lemma} and Lemma \ref{Lem: integral quotient} imply that $\calq_{x_{as}}$ is a connected smooth scheme over $\calo_F$.

Recalling the symmetrization map at $x_{as}$ $s_{x_{as}}: \G_{x_{as}}\lra \calq$ given by 
\[
s_{x_{as}}(g) = g{x_{as}}\sig(g)= {x_{as}} s(g).
\]
By Proposition \ref{Prop: relative Kazdhan lemma} and Corollary \ref{Cor: absolute nice lifts}, this gives a closed embedding of $\calo_F$-schemes
\begin{align*}
    s_{x_{as}}:\calq_{x_{as}}&\lra \calq\\
    x&\longmapsto   {x_{as}} x=x{x_{as}}.
\end{align*}
In particular, $s_{x_{as}}(x_{tu}) = x$. In this section, we will relate orbital integrals on $\calq(F)$ with orbital integrals on $\calq_{x_{as}}(F)$ via descent. Strictly speaking, the descent is not to $\calq_{x_{as}}(F)$, but its image under $s_{x_{as}}$. The next lemma says that the these give the same value.
\begin{Lem}\label{Lem: all good}
 Suppose that $x=x_{as}x_{tu}\in \calq(\calo_F)$ as above. Then
 \[
 \displaystyle\int_{\rH_{{x_{as}},x}(F)\backslash \rH_{x_{as}}(F)}\bfun_{s_{{x_{as}}}(\calq_{{x_{as}}}(\calo_F))}(h^{-1}xh)dh = \displaystyle\int_{\rH_{{x_{as}},x_{tu}}(F)\backslash \rH_{x_{as}}(F)}\bfun_{\calq_{x_{as}}(\calo_F)}(h^{-1}x_{tu}h)dh. 
 \]
  Furthermore,
 \[
 e(\rH_{x})=e(\rH_{{x_{as}},x_{tu}}).
 \]
\end{Lem}
Note that this later integral is nothing but the $\rH_{x_{as}}(F)$-orbital integral $\Orb(x_{tu},\bfun_{\calq_{x_{as}}(\calo_F)})$ on $\calq_{{x_{as}}}(F)$.
\begin{proof}
This is immediate at ${x_{as}}\in \calq(\calo_F)$ is central for $\rH_{x_{as}}$ and integral. Furthermore,  Lemma \ref{Lem: top nil reg rel} ensures that $x_{tu}\in\calq_{{x_{as}}}^{rss}(F)$ and  
\[
\rH_{{x_{as}},x}=\rH_{{x_{as}},x_{tu}}.
\]
The identity of Kottwitz signs also follows.
\end{proof} 

Recall that we have fixed measures $dh$ on $\rH(F)$, $dt$ on $\rH_x(F)$, and $dh_{x_{as}}$ on $\rH_{{x_{as}}}(F)$ so that our chosen hyperspecial maximal compact subgroups have volume $1$. In particular, both $\calq(\calo_F)$ and $\calq_{{x_{as}}}(\calo_F)$ have volume one with these choices. 

\begin{Prop}\label{Prop: absolute descent}
Let  $\ka\in \mathfrak{C}(\rH_x,\rH;F)^D$ and let $\ka_{{x_{as}}}$ denote the restriction of the character $\ka$ to 
\[
\mathfrak{C}(\rH_x,\rH_{x_{as}};F)\subset\mathfrak{C}(\rH_x,\rH;F).
\]

Under the above assumptions, we have the identity
\[
\Orb^\ka(x,\bfun_{\calq(\calo_F)}) = \Orb^{\ka_{{x_{as}}}}(x_{tu},\bfun_{\calq_{{x_{as}}}(\calo_F)}).
\]
\end{Prop}
\begin{proof}
Our choice of $x\in \calq(\calo_F)$ fixes a bijection
\begin{align*}
    \left\{\text{$\rH(F)$-orbits in  } (\rH_x\backslash \rH)(F)\right\}&\iso \cald(\rH_x,\rH;F)\\
    x'\qquad\qquad\quad&\mapsto \inv(x,x').
\end{align*}
For each rational orbit $\rH(F)\cdot x'$, we decompose this with respect to the action of $\mathcal{H}(\calo_F)$
\[
\rH(F)\cdot x' = \bigsqcup_{y}\mathcal{H}(\calo_F)\cdot y
\]
with $y\in \rH(F)\cdot x'$ running over a set of $\mathcal{H}(\calo_F)$-orbit representatives and set 
\[
\# \mathcal{H}(\calo_F)[x']:=\#\{y\in \mathcal{H}(\calo_F)\backslash\rH(F)\cdot x': y\in \calq(\calo_F)\}
\]
to be the number of $\mathcal{H}(\calo_F)$-orbits in $\rH(F)\cdot x'$ with an integral representative. With our choices of measure, we see that
\[
\Orb^\ka(x,\bfun_{\calq(\calo_F)}) = \sum_{x'}e(\rH_{x'})\ka(\inv(x,x'))\#\mathcal{H}(\calo_F)[x'].
\]

Now if $x'=x_{as}'x_{tu}'\in \calq(\calo_F)$ lies in the stable orbit of $x$, there is an $h\in \rH(\Fbar)$ such that
\[
x_{as}x_{tu}=hx_{as}'x_{tu}'h^{-1} = (hx_{as}'h^{-1})(hx_{tu}'h^{-1}).
\] It is not immediately clear that $hx_{as}'h^{-1}$ and $hx_{tu}'h^{-1}$ are $F$-rational; nevertheless, if we pass to a finite extension $L/F$ such that $hx_{as}'h^{-1},hx_{tu}'h^{-1}\in \G(L)$, uniqueness of the topological Jordan decomposition forces
\[
x_{as}=hx_{as}'h^{-1}\text{ and }x_{tu}=hx_{tu}'h^{-1},
\]which a postiori implies the elements are $F$-rational. In particular, if $x$ and $x'$ lie in the same stable orbit, the same holds for their absolutely semi-simple and topologically nilpotent parts.

Applying Proposition \ref{Prop: relative Kazdhan lemma} now to ${x_{as}}$ and $x_{as}'$, we may conjugate $x'$ by an element of $\mathcal{H}(\calo_F)$ to assume that ${x_{as}}=x_{as}'$. This implies that $x'\in \G_{{x_{as}}}(F)$ and if $x=hx'h^{-1},$ then $h\in \rH_{x_{as}}(\Fbar)$. In particular, $x$ and $x'$ lie in the same stable $\rH_{x_{as}}$-orbit in $s_{{x_{as}}}(\calq_{{x_{as}}}(F))$. This implies that
\[
\inv(x,x')\in \cald(\rH_x,\rH_{x_{as}};F)\subset \mathfrak{C}(\rH_x,\rH_{x_{as}};F).
\]

Furthermore, we claim that
\[
\# \mathcal{H}(\calo_F)[x']=\# \mathcal{H}_{x_{as}}(\calo_F)[x'],
\]
where the right-hand side counts the number of $\mathcal{H}_{x_{as}}(\calo_F)$-orbits in $\rH_{x_{as}}(F)\cdot x'$ with an integral representative. To see this, assume $y\in\mathcal{H}(\calo_F)[x']$ represents an $\mathcal{H}(\calo_F)$-orbit with $y\in \calq(\calo_F)$. Then $hyh^{-1}=x'$ for some $h\in \rH(F)$, and if $y=y_{as}y_{tn}$ is the associated decomposition, then $y_{as}$ lies in the same $\rH(F)$-orbit of $x_{as}$. Now Proposition \ref{Prop: relative Kazdhan lemma} forces 
\[
h\in \mathcal{H}(\calo_F)\rH_{x_{as}}(F).
\]
Thus, up to changing $y$ by an element of $\mathcal{H}(\calo_F)$ to ensure that $h\in \rH_{x_{as}}(F)$, we have
\[
\mathcal{H}_{x_{as}}(\calo_F)\cdot y\subset \rH_{x_{as}}(F)\cdot x'.
\]
On the other hand, suppose $y'\in \mathcal{H}(\calo_F)\cdot y$ is another element such that $h'y'(h')^{-1}=x'$ for some $h'\in \rH_{x_{as}}(F)$. Since $hyh^{-1}=x'$ for $h\in \rH_{x_{as}}(F)$, we see that
\[
y_{as}=x_{as}'=y_{as}',
\]
forcing $y'\in \mathcal{H}_{x_{as}}(\calo_F)\cdot y$. Thus, the intersection of $\mathcal{H}(\calo_F)\cdot y$ with $\rH_{x_{as}}(F)\cdot x'$ consists of a single $\mathcal{H}_{x_{as}}(\calo_F)$-orbit. This sets up a bijection between the two sets of orbits and the claim follows. 

Using our normalization of measures, Lemma \ref{Lem: all good} implies the result.
\end{proof}


\section{The symmetric space and invariant theory}\label{Section: symmetric}
We now specialize to the case of primary interest. Let $W_1$ and $W_2$ be two $n$-dimensional Hermitian spaces with respect to our fixed quadratic extension of $p$-adic $E/F$. Set $W=W_1\oplus W_2$ be the $2n$-dimensional Hermitian space. Let $\ep\in \U(W)$ be an element of order $2$ inducing the eigenvalue decomposition $W=W_1\oplus W_2$. We then have the involution $\theta(g)=\ep g\ep$ on both $\U(W)$ and $\GL(W)$, with corresponding fixed-point subgroups 
 \[
 \U(W_1)\times \U(W_2)\subset \U(W)\text{ and }\GL(W_1)\times \GL(W_2)\subset \GL(W).
 \]
Set $\G=\U(W)$ and $\rH= \U(W_1)\times \U(W_2)$. In this section, we study the associated symmetric space
\[
\calq(F):=U(W)/ U(W_1)\times U(W_2).
\]

\subsection{The linear symmetric space}
For the purposes of invariant theory, it is useful to first consider the base change of the variety $\calq$ to $E$, the $F$-points of which are isomorphic to 
\[
\cals(F)=\Res_{E/F}(\calq_E)(F) =\GL(W)/\GL(W_1)\times \GL(W_2).
\]

Consider the symmetrization map $s:\GL(W)\lra \cals$. Since $H^1(F,\GL(W_1)\times\GL(W_2))=0$, we have a surjection on $F$-points and an identification
\[
s:\GL(W)/\GL(W_1)\times \GL(W_2)\iso \cals(F).
\]Given an element $x=s(g)\in \cals(F)$, write
$
x= \left(\begin{array}{cc}A&B\\C&D\end{array}\right).
$ Then the block matrices satisfy the polynomial relations
\[
A^2=I_n+BC,\: D^2=I_n+CB,\: AB=BD,\: CA=DC.
\]
These relations are not sufficient to cut out $\cals$, cutting out instead the subvariety $\G^\sig\subset \GL(W)$ of elements satisfying $\sig(g) = g$. Unwinding the definition, this is the variety of $x\in \GL(W)$ such that $\ep x$ is an involution. We have a decomposition of $\G^\sig$ into irreducible components
\[
\G^\sig=\bigsqcup_{i=1}^{2n}\G^\sig_i
\]
where if for any $x\in \G^\sig(F)$, we have an eigenspace decomposition 
\[
W=W_{x,1}\oplus W_{x,-1},
\]
for the involution $\ep x$ and
\begin{equation}\label{eqn: components}
    \G^\sig_i(F) = \{x\in \G^\sig(F): \dim(W_{x,-1}) = i\}.
\end{equation}

In general, a computation of the characteristic polynomial of $\ep x$ distinguishes these components. It is clear that $\cals= \G^\sig_n,$ since $\ep s(g)= \ep g\ep g^{-1} \ep$ is conjugate to $\ep$. 
\begin{Lem}\label{Lem: minus sign}
An element $x\in \G^\sig(F)$ lies in $\G^\sig_i(F)$ if and only if $-x\in \G^\sig_{2n-i}(F)$. In particular, $x\in \cals(F)$ if and only if $-x\in \cals(F)$.
\end{Lem}
\begin{proof}
 Since multiplication by $-I_{2n}$ multiplies all the eigenvalues of the involution $\ep x$ by $-1$, the lemma immediately follows from \eqref{eqn: components} and the identification $\cals= \G^\sig_n$.
\end{proof}
Realizing the embedding $\GL(W_1)\times \GL(W_2)\subset \GL(W)$ in a block-diagonal fashion, the action of $(g,h)\in \GL(W_1)\times \GL(W_2)$ on $x\in \cals(F)$ is given by
\[
(g,h)\cdot x =\left(\begin{array}{cc}gAg^{-1}&gBh^{-1}\\hCg^{-1}&hDh^{-1}\end{array}\right).
\]
The following lemma gives a nice set of orbit representatives of the semi-simple elements of $\cals(F)$.
\begin{Lem}\cite[Lemma 4.3]{JRUniqueness}\label{Lem: orbit reps}
 Each semi-simple element $x\in \cals^{ss}(F)$ is $\GL(W_1)\times \GL(W_2)$-conjugate to an element of the form
\begin{equation}\label{Prop: GL orbit reps}
x(A,n_1,n_2):=\left(\begin{array}{cccccc}A&0&0&A-I_m&0&0\\0&I_{n_1}&0&0&0&0\\0&0&-I_{n_2}&0&0&0\\A+I_m&0&0&A&0&0\\0&0&0&0&I_{n_1}&0\\0&0&0&0&0&-I_{n_2}\end{array}\right),
\end{equation}
with $n=m+n_1+n_2$ and $A\in \fgl_m(E)$ semi-simple without eigenvalues $\pm1$ and unique up to conjugation. Moreover, $x(A,n_1,n_2)$ is regular if and only if $n_1=n_2=0$ and $A$ is regular in $\fgl_n(E)$.
\end{Lem}
This has the following simple consequence. 
\begin{Lem}\label{Lem: characteristic comparison}
Suppose that
\[
x= \left(\begin{array}{cc}A&B\\C&D\end{array}\right)\in \cals(F).
\]
The characteristic polynomial of $x$ is 
\[
\car_x(t) = \det(t^2I_n-2tA+I_n).
\]
In particular, if we let $\chi_x(t):=\det(tI_n-A)$ denote the characteristic polynomial for the $n\times n$ matrix $A$, the eigenvalues of $x$ are given by
\[
\Omega(x):=\{\al\pm\sqrt{\al^2-1}:\al\text{ is a root of }\chi_x(t)\};
\]
here we abuse notation in the non-archimedean setting and use $\pm\sqrt{a}$ to denote the two square roots of $a\in F^{alg}$ as a set.
\end{Lem}
\begin{proof}
It suffices to prove the claim on the Zariski-open and dense set $\cals^{rss}(F)$ of regular semi-simple elements. Using the previous lemma, this amounts to the following identity
\begin{align*}
   \det\left(\begin{array}{cc}
    tI_n-A &I_n-A  \\
    -I_n-A & tI_n-A 
\end{array}\right)&=\det\left((tI_n-A)^2-(I_n-A)(-I_n-A)\right)\\
                &=\det\left(t^2I_n-2tA+I_n\right).
\end{align*}
Here we apply the block-matrix determinant identity
\[
\det\left(\begin{array}{cc}
    A & B \\
    C & D
\end{array}\right)=\det(AD-BC),
\] which holds whenever $CD=DC$.
\end{proof}
 Consider now the $\GL(W_1)\times \GL(W_2)$-invariant map $\chi:\cals\to \A^n$ given by sending $x\in \cals(F)$ to the coefficients of the monic polynomial $\chi_x(t) =\det(tI_n-A)$.
\begin{Lem}\label{Lem: cat quotient}
The pair $(\A^n,\chi)$ is a categorical quotient for $(\GL(W_1)\times \GL(W_2),\cals)$.
\end{Lem}
\begin{proof}
As the statement is geometric, we may assume that $E=E^{alg}$. We make use of Igusa's criterion \cite[Section 3]{ZhangFourier}: let a reductive group $H$ act on an irreducible affine variety $X$. Let $Q$ be a normal irreducible variety, and let $\pi:X\to Q$ be a morphism that is constant on $H$ orbits such that
\begin{enumerate}
\item $Q-\pi(X)$ has codimension at least two,
\item there exists a nonempty open subset $Q'\subset Q$ such that the fiber $\pi^{-1}(q)$ of $q\in Q'$ contains exactly one orbit.
\end{enumerate}
Then $(Q,\pi)$ is a categorical quotient of $(H,X)$. 

To show this, we use Lemma \ref{Lem: orbit reps}. Given $n$-tuple $(a_1,\ldots, a_n)$, one may form the polynomial
\[
p_{(a_i)}(t)=t^n+a_1t^{n-1}+\cdots+a_n = \displaystyle\prod_{i=1}^m(t-\al_i)\times(t-1)^{n_1}(t+1)^{n_2},
\]
for certain $\al_i\in E$. Setting $A=\diag(\al_1,\ldots,\al_m)$, we see that $\chi(x(A,n_1,n_2)) = (a_1,\ldots, a_n)$ so that $\chi$ is surjective. This shows the first requirement of Igusa. 

For the second, we consider the subset $Q'\subset \A^n$ given by
 $$Q'=\{(a_1,\ldots, a_n):p_{(a_i)}(t)=\det(tI_n-A)\text{ for some }A\in \fgl^{rss}_n(E)-(D_1\cup D_{-1})(E)\},$$
 where for any $a\in E$, $D_a=\{X\in \fgl_n: \det(aI_n-X)=0\}$. Lemma \ref{Lem: orbit reps} shows that $Q'$ is the image of the \emph{regular semi-simple locus} of $\cals$. On the other hand, if $x\in \cals(F)$ such that $\chi(x)\in Q'$, Lemma \ref{Lem: characteristic comparison} implies that $x$ is regular semi-simple as an element of $\GL(W)$. In particular, $x\in \cals^{rss}(F)$. The uniqueness statement of Lemma \ref{Lem: orbit reps} now implies that there is a unique orbit in the fiber over the coefficients of $\det(tI_n-A)$. This implies the second criterion for the open set $Q'$.
\end{proof}
\begin{Rem}
In the process of our proof, we showed that $\cals^{rss}\subset \GL(W)^{rss}$. This shows that this symmetric space is quasi-split (see (\ref{eqn: quasi-split condition})). As this notion is geometric, this implies that $\calq=\U(W)/\U(W_1)\times\U(W_2)$ is also quasi-split.
\end{Rem}

A similar argument gives the following lemma for the quotient by the $\GL(W_2)$-factor.
\begin{Lem}\label{Lem: halfway}
Let $R: \cals\to \fgl(W_1)$ denote the map 
\[
\left(\begin{array}{cc}A&B\\C&D\end{array}\right)\longmapsto A.
\]
Then $(\fgl(W_1), R)$ is a categorical quotient for the $GL(W_2)$-action on $\cals$ given by
\[
h\cdot \left(\begin{array}{cc}A&B\\C&D\end{array}\right)= \left(\begin{array}{cc}A&Bh^{-1}\\hC&hDh^{-1}\end{array}\right).
\]
The map $R$ is $\GL(W_1)$-equivariant with respect to the adjoint action on $\fgl(W_1)$.
\end{Lem}
\begin{proof}
The set of $\GL(W_1)\times \GL(W_2)$-orbit representatives given by (\ref{Prop: GL orbit reps}) shows that the map $R$ is surjects onto the semi-simple locus of $\fgl(W_1)$. Indeed, it is a simple exercise that
\[
\left(\begin{array}{cc}
    A &I_n  \\
    A^2-I_n & A 
\end{array}\right)\in \cals(F)
\] for any $A\in \fgl(W_1)$. This shows $R$ is surjective, proving the first criterion of Igusa.

Let $Q'=\fgl(W_1)^{rss}$. Then as before, if $x\in \cals(F)$ satisfies $R(x) \in Q'$, Lemma \ref{Lem: characteristic comparison} implies that $x\in \cals^{rss}(F)$. The uniqueness statement of Lemma \ref{Lem: orbit reps} thus implies that
\[
x=(g,h)\cdot x(A,0,0)
\]for some $h\in \GL(W_1)$ and for some $g\in \GL(W_1)_A$. However, for such $g$ we have
\[
\left(\begin{array}{cc}
    gAg^{-1} &g(A-I_n)h^{-1}  \\
    h(A+I_n)g^{-1} & hAh^{-1} 
\end{array}\right)=\left(\begin{array}{cc}
    A &(A-I_n)(hg^{-1})^{-1}  \\
    hg^{-1}(A+I_n) & (hg^{-1})A(hg^{-1})^{-1} 
\end{array}\right),
\]
so that $x$ is in the same $\GL(W_2)$-orbit as $x(A,0,0)$. The second criterion of Igusa thus follows.
\end{proof}

\subsection{The unitary symmetric space}\label{Section: unitary orbits}
We now account for various arithmetic aspects of $\calq$. We have the exact sequence of pointed sets
\begin{equation}\label{eqn: long exact sequence}
1\to U(W_1)\times U(W_2)\to U(W)\to \calq(F)\to \ker\left[H^1(F,\U(W_1)\times \U(W_2))\to H^1(F,\U(W))\right].
\end{equation}

\begin{Lem}\label{Lem: Forms in first quotient}
Let $E/F$ be a quadratic extension of $p$-adic fields. There exist two isomorphism classes of decomposition
\[
W_1\oplus W_2=W=W_1'\oplus W_2'.
\]We have a bijection of $F$-points
\begin{equation*}
\calq(F) = U(W)/U(W_1)\times U(W_2)\bigsqcup U(W)/U(W'_1)\times U(W'_2)
\end{equation*}
where the first quotient is identified with the image of $s:U(W)\to \calq(F)$. 
\end{Lem}
\begin{proof}
This is a basic Galois cohomology calculation. We omit the details.
\end{proof}


We pause to introduce some notation. The symmetrization map takes the form $s(g)=g\ep g^\dagger \ep$, where $\dagger$ denotes the adjoint map such that
\[
U(W)=\{g\in \GL(W): gg^\dagger = I_W\}.
\]
Writing this out, we have
\[
s(g)=\left(\begin{array}{cc}AA^\ast-BB^\ast&CA^\ast-DB^\ast\\BD^\ast-AC^\ast&DD^\ast-CC^\ast\end{array}\right) \qquad\text{ for $g= \left(\begin{array}{cc}A&B\\C&D\end{array}\right)$.}
\]
Here we need to be precise about the overloaded notation. For $A\in \End(W_1)$, $A^\ast$ is the adjoint operator with respect to the Hermitian form $\Phi_1$ on $W_1$:
\[
\la Av,w\ra_1= \la v,A^\ast w\ra_1\qquad\text{ for all } v,w\in W_1;
\]similarly with $D\in \End(V_2)$. For $B\in \Hom_{E}(W_2,W_1)$, the endomorphism $B^\ast\in\Hom_{E}(W_1,W_2)$ is defined by 
\[
\la Bv,w\ra_1=\la v,B^\ast w\ra_2,\qquad\text{ for all }w\in W_1,\:v\in W_2;
\]
the map $C\mapsto C^\ast$ is analogous. In particular, any element $x\in \calq(F)$ may be written
\[
x=\left(\begin{array}{cc}A&B\\-B^\ast&D\end{array}\right),
\]
where $A\in \Herm(W_1)$, $D\in \Herm(W_2)$, and $B\in \Hom_{E}(W_2,W_1)$. As before, the blocks satisfy the polynomial relations
\[
A^2=I_n-BB^\ast,\: D^2=I_n-B^\ast B,\: AB=BD, B^\ast A= DB^\ast.
\]
As in the linear case, we define the morphism $\chi:\calq\to \A^n$ by sending $x$ to the coefficients of the monic polynomial $\chi_x(t)= \det(tI-A)$. 
\begin{Lem}\label{Lem: cat quotient unitary}
The pair $(\A^n,\chi)$ is a categorical quotient for the $\U(W_1)\times \U(W_2)$-action on $\calq$.
\end{Lem} 
\begin{proof}
As the assertion is geometric, we may prove this over the algebraic closure, at which point the claim follows from Lemma \ref{Lem: cat quotient}.
\end{proof}
\subsubsection{The contraction map}
We also have a unitary version of Lemma \ref{Lem: halfway}:
\begin{Lem}\label{Lem: unitary halfway}
Define the \emph{contraction map} $R:\calq\to \Herm(W_1)$ given by $$\left(\begin{array}{cc}A&B\\-B^\ast&D\end{array}\right)\longmapsto A.$$ The pair $(\Herm(W_1),R)$ is a categorical quotient for the $\U(W_2)$-action on $\calq$.
\end{Lem}

A useful consequence of the orbit computations (\ref{Prop: GL orbit reps}) is that if $x\in \calq(F)$ is regular semi-simple, then $\det(B)\neq0$ and $A$ is regular semi-simple in $\Herm(W_1)$. Let $\calq^{iso}\subset \calq$ denote the Zariski-open subvariety cut out by this determinant condition. The superscript $iso$ refers to the fact that 
\[
x\in \calq^{iso}(F)\:\text{ if and only if }\: I-R(x)^2 \in \mathrm{Iso}(W_1,W_1).
\]
Setting \[
\Herm(W_1)^{iso}=\{A\in \Herm(W_1): I-A^2\text{ is non-singular}\},
\]
we obtain a map $R:\calq^{iso}\lra \Herm(W_1)^{iso}$. 

\begin{Lem}\label{Lem: centralizer contraction}
The restriction $R:\calq^{iso}\to \Herm(W_1)^{iso}$ is a $\U(W_2)$-torsor. Moreover, for $x\in \calq^{iso}(F)$, we have an isomorphism
\[
H_x\iso \U(W_1)_{R(x)}
\]
given by $(h_1,h_2)\mapsto h_1$.
\end{Lem}

\begin{proof}This is analogous to Lemma 3.6 of \cite{Leslieendoscopy}, and is proved in the same way.
\end{proof}



\begin{Lem}\label{Lem: image of contraction}
Identify $H^1(F,\U(W_2))=F^\times/\Nm_{E/F}(E^\times)$ via the discriminant map $$(W_2,\Phi)\mapsto d(\Phi)\in F^\times/\Nm_{E/F}(E^\times),$$ where 
\[
d(\Phi):=(-1)^{n(n-1)/2}\det(\Phi).
\]Then $X$ is in the image of $R:\calq^{iso}(F)\lra \Herm(W_1)^{iso}$ if and only if
\[
d(I-X^2)\equiv d(\la\cdot,\cdot\ra_1)\cdot d(\la\cdot,\cdot\ra_2)\pmod{\Nm_{E/F}(E^\times)}
\]
\end{Lem}
\begin{proof}
 The claim follows from the definition of the map $R$ in Lemma \ref{Lem: unitary halfway} and the relation $I_n-A^2=BB^\ast$.
\end{proof}

The inclusion $\calq^{rss}\subset \calq^{iso}$ and Lemma \ref{Lem: centralizer contraction} imply that the restriction of the contraction map $R: \calq\lra \Herm(W_1)$ to the regular semi-simple locus is a $\U(W_2)$-torsor. 
The next lemma enables us to use $R$ to study $\ka$-orbital integrals at regular semi-simple elements of $\calq(F)$ in the next subsection.

\begin{Lem}\label{Lem: kernel isom}
Let $x\in \calq^{rss}(F)$ and set $R(x)=y\in \Herm(W_1)$. The isomorphism 
\[
\phi: \rH_x\iso \U(W_1)_y
\]induces an isomorphism 
\begin{equation*}
\mathfrak{C}(\rH_x,\rH;F)\iso \mathfrak{C}(\U(W_1)_y,\U(W_1);F).
\end{equation*}
\end{Lem}
\begin{proof}
This is analogous to the Lie algebra version \cite[Lemma 3.9]{Leslieendoscopy}; we include a slightly more hands-on argument afforded by our restriction to the $p$-adic setting. In particular, we may use the fact that 
\[
\mathfrak{C}(\rH_x,\rH;F)\cong \cald(\rH_x,\rH;F)
\] in this setting to prove the lemma using the invariant map. Consider the commutative diagram
\[
\begin{tikzcd}
H^1(F,\rH_x)\ar[r,"\iota_x"]\ar[d,"\phi^\ast"]& H^1(F,\U(W_1))\times H^1(F,\U(W_2))\ar[d,"p_1^\ast"]\\
H^1(F,\U(W_1)_y)\ar[r,"\iota_y"]&H^1(F,\U(W_1)).
\end{tikzcd}
\]
where $p_1:\U(W_1)\times \U(W_2)\lra \U(W_1)$ is the projection, $\phi: \rH_x\iso \U(W_1)_y$ is the induced isomorphism, and $\phi^\ast$ and $p_1^\ast$ are the maps induced on cohomology.

If $\al\in\cald(\rH_x,\rH;F)$, then $\iota_\de\phi(\al)=p_1(\iota_x(\al))=1$. This allows us to extend the diagram to
\[
\begin{tikzcd}
1\ar[r]&\cald(\rH_x,\rH;F)\ar[r]\ar[d]&H^1(F,\rH_x)\ar[r,"\iota_x"]\ar[d,"\phi^\ast"]& H^1(F,\U(W_1))\times H^1(F,\U(W_2))\ar[d,"p_1^\ast"]\\
1\ar[r]&\cald(\U(W_1)_y,\U(W_1);F)\ar[r]&H^1(F,\U(W_1)_y)\ar[r,"\iota_\de"]&H^1(F,\U(W_1)),
\end{tikzcd}
\]
where the arrow $\cald(\rH_x,\rH;F)\to \cald(\U(W_1)_y,\U(W_1);F)$ is an injection. To show it is surjective, we show that it induces a surjection on rational orbits in the given stable orbit. Suppose that $y'\in \Herm(W_1)$ is stably conjugate to $y$; this gives the element
\[
\inv(y,y')=[\tau\longmapsto \tau(h)^{-1}h]\in \cald(\U(W_1)_y,\U(W_1);F)
\]
where $h\in \U(W_1)(\Fbar)$ such that $y'=hyh^{-1}$. Since $R(x) =y$ and 
\[
d(I-(y')^2)\equiv d(I-y^2)\pmod{\Nm_{E/F}(E^\times)},
\]
Lemmas \ref{Lem: centralizer contraction} and \ref{Lem: image of contraction} combine to imply that there exists $x'\in \calq^{rss}(F)$ such that $R(x') = y'$.

Then
\[
R(x') = y'=hyh^{-1}=R\left(\left(\begin{array}{cc}
    h &  \\
     & I
\end{array}\right)\cdot x\right).
\]
The $\U(W_2)$-torsor statement of Lemma \ref{Lem: centralizer contraction} now implies that there exists $h'\in \U(W_2)(\Fbar)$ such that
\[
\left(\begin{array}{cc}
   I &  \\
     &  (h')^{-1}
\end{array}\right)\cdot x'=\left(\begin{array}{cc}
    h &  \\
     & I
\end{array}\right)\cdot x\Longrightarrow x'=\left(\begin{array}{cc}
    h &  \\
     & h'
\end{array}\right)\cdot x;
\]
that is, $x'$ is stably conjugate to $x$ and 
\[\inv(x,x') = [\tau\longmapsto (\tau(h)^{-1}h,\tau(h')^{-1}h')]\in \cald(\rH_x,\rH;F)
\]
maps to $\inv(y,y')\in \cald(\U(W_1)_y,\U(W_1);F)$.
\end{proof}


\subsection{Descendants}
 For this subsection only, we let $E/F$ denote a quadratic extension of fields of either odd or zero characteristic. 
 
We compute the possible descendants $(\G_x,\rH_x)$ of $(\G,\rH)$ at semi-simple points $x\in \calq^{ss}(F)$. An important corollary of this computation is that all the stabilizers $H_x\subset H$ are connected reductive groups (see Lemma \ref{Lem: unitary nice})\footnote{This fact already follows over the algebraic closure from the orbit computation (\ref{Prop: GL orbit reps}).}. We remark that our descent argument in Section \ref{Section: descent final} only encounters descendants of the form (\ref{1}) and (\ref{2}) below. Regardless, the general form will be useful for later applications toward smooth transfer.

\begin{Lem}\label{Lem: descendants}
For any $x\in \calq^{ss}(F)$, there is an orthogonal decomposition of $W$
\[
W=V_0\oplus V_1\oplus V_{-1},
\]
with $V_1$ (resp. $V_{-1}$) is the $1$-eigenspace (resp. $-1$-eigenspace) of $x$ and $V_0$ is the orthogonal compliment of $V_1\oplus V_{-1}$ in $W$. The involution $\theta(g)=\ep g\ep$ preserves this decomposition, and the symmetric pair $(U(W)_x,(U(W_1)\times U(W_2))_{x})$ is a product of the following symmetric pairs:\\

\begin{enumerate}
    \item\label{1} $(U(V_1),U(V_{1,1})\times U(V_{1,-1}))$, where $V_1$ is the $1$-eigenspace of $x$, and 
    \[
    V_{1,\pm1}=\{v\in V_1: \ep v=\pm v\};
    \]
     \item\label{2} $(U(V_{-1}),U(V_{-1,1})\times U(V_{-1,-1}))$, where $V_{-1}$ is the $-1$-eigenspace of $x$, and 
     \[
     V_{-1,\pm1}=\{v\in V_{-1}: \ep v=\pm v\};
     \]
    \item\label{case3}\label{descendants1} $(\GL(V'), U(V'))$, where $V'$ is a non-degenerate Hermitian space over $E'/F'$. Here, $F'$ is a finite extension of $F$ and $E'=EF'$ is the associated quadratic extension;\\
    \item\label{descendants2} $(U(V')\times U(V'), U(V'))$, with $U(V')$ embedded diagonally;\\
    \item\label{case5} $(\GL(V')\times \GL(V'), \GL(V'))$, with $\GL(V')$ embedded diagonally.
\end{enumerate}

\end{Lem}
\begin{proof}
We begin by decomposition 
\[
W=\bigoplus_i V_i
\]
where each $V_i$ is a subspace upon which the minimal polynomial of $x|_{V_i}$ is irreducible. For each $i$, let $E_i$ be the finite extension of $E$ cut out by $x$; we have then $V_i \cong E_i^{n_i}$ for some $n_i$. We set $\al_i\in E_i^\times$ for the eigenvalue of $x$ on $E_i$. 

Let $P(t)=\car_x(t)$ denote the characteristic polynomial of $x$ and let $P_i$ denote the minimal polynomial of $x|_{V_i}$. Noting that 
\[
x\in U(W)\implies x^\dagger=x^{-1},
\]
we have $P(t) = \frac{t^{\dim(W)}}{\overline{P}(0)}\overline{P}(t^{-1})$, where $\overline{P}(t)$ denotes the action of the non-trivial Galois element of $\Gal(E/F)$ on the coefficients. This implies a product decomposition
\[
P(t)=\prod_{i\in I} P_i(t)^{n_i}\prod_{(j,j')\in J}\left(P_j(t)P_{j'}(t)\right)^{n_j},
\]
where for each $i\in I$,
\[
P_i(t) = \frac{t^{\dim(V_i)}}{\overline{P}_i(0)}\overline{P}_i(t^{-1}),
\]
and for each pair $(j,j')\in J$
\[
P_{j'}(t) = \frac{t^{\dim(V_j)}}{\overline{P}_j(0)}\overline{P}_j(t^{-1}).
\]
Thus, for each $i\in I$, we obtain a Galois element $\overline{(\cdot)}:E_i\lra E_i$ induced by
\begin{align*}
E_i\cong E[t]/(P_i(t))&\lra E[t]/(\overline{P}_i(t))\cong E_i\\
\al_i\longmapsto t&\longmapsto t\longmapsto \al_i^{-1}.
\end{align*}
Setting $F_i$ to be the field fixed by this involution, we obtain a quadratic extension $E_i/F_i$ and note that $\overline{\al_i} = \al_i^{-1}$. It is now easy to see that the Hermitian form on $W$ restricts to a non-degenerate Hermitian form on $V_i$ with respect to the quadratic extension $E_i/ F_i$.

For each $(j,j')\in J,$ a similar argument shows an isomorphism $E_j\iso E_{j'}$. Under the identification, we find that $\al_{j'}=\overline{\al}_j^{-1}$, and the restriction of the Hermitian form on $W$ to $V_j\oplus V_{j'}$ is non-degenerate, the direct summands of the decomposition being maximal isotypic subspaces. Thus, we obtain the product 
\[
U(W)_x=\prod_{i\in I} U(V_i)\times \prod_{(j,j')\in J}\GL(V_j).
\]

We now compute the group $(U(W_1)\times U(W_2))_x$. For simplicity, fix $i$ and set $E'=E_i$, $V'=V_i$, and let $\al=\al_i$ denote the associated eigenvalue. Since $\ep x \ep =x^{-1}$, we see that $\ep( V')$ is the $\al^{-1}$-eigenspace. In particular, $\ep$ fixes $V'$ if and only if $\al=\al^{-1}$. Set
\[
C(\al)=\{\al,{\al}^{-1},\overline{\al},\overline{\al}^{-1}\}.\\
\]

\noindent
\textbf{\underline{Case 1:}} $C(\al)=\{\al\}$.\\

In this case, $\al=\al^{-1},$ so that $\al=\pm1$. Clearly, $E'=E$ and $\ep (V') =V'$. This induces an eigenvalue decomposition 
\[
V'= V'_1\oplus V'_{-1}.
\]
A simple exercise shows that the restriction of the Hermitian pairing is non-degenerate on each eigenspace, so we obtain the symmetric pair 
\[
(U(V'),U(V'_1)\times U(V'_{-1})).
\]
It follows from the orbit representatives in Lemma \ref{Lem: orbit reps} that $\dim(V')$ is even and that $\dim(V'_1)=\dim(V'_{-1})$.\\

\noindent
\textbf{\underline{Case 2:}} $C(\al)=\{\al,\overline{\al}\}$.\\

In this case, $\al =\overline{\al}^{-1}$ but $\al\neq \overline{\al}$. Then $\ep (V')$ is the $\overline{\al}$-eigenspace, and we find the symmetric pair
\[
(U(V')\times U(\ep V'), U(V')),
\]
with respect to the embedding $g\mapsto (g, \theta(g))$.\\

\noindent
\textbf{\underline{Case 3:}} $C(\al)=\{\al,\al^{-1}\}$.\\

In this case, $\al = \overline{\al}$, so that $\ep (V')$ is the $\overline{\al}^{-1}$-eigenspace. This produces the pair
\[
(\GL(V'),U(V'')),
\]
where 
\begin{align*}
    V''=\{(w,\ep w): w\in V'\}\hra V'\oplus \ep (V').
\end{align*}
The projection 
\[
\begin{tikzcd}
U(V'')\ar[r]\ar[dr]&\GL(V')\times \GL(\ep V')\ar[d,"p_1"]\\
&\GL(V')
\end{tikzcd}
\]produces an embedding $U(V'')\hra \GL(V')$ where the resulting form on $V'$ is given by $\la w,\ep v\ra$, for $w,v\in V'$.\\

\noindent
\textbf{\underline{Case 4:}} $C(\al)=\{\al,{\al}^{-1},\overline{\al},\overline{\al}^{-1}\}$.\\

Finally, we have the case that all eigenvalues are distinct. Then $\ep( V')$ is the $\al^{-1}$-eigenspace, and the spaces $V'$ and $\ep (V')$ belong to distinct pairs $(V_j,V_{j'})$ with $(j,j')\in J$. Thus, we have the pair
\[
(\GL(V')\times \GL(\ep V'), \GL(V'))
\]
with the embedding $g\mapsto (g,\theta(g))$. This exhausts the cases and establishes the lemma.
\end{proof}
 
As a corollary, we now show that our symmetric space admits a nice integral model in the unramified setting in the sense of Section \ref{Section: relative kazhdan}. We therefore assume that $G=\U(V_n\oplus V_n)$ and $H= \U(V_n)\times \U(V_n),$ where $V_n$ is a split Hermitian space of dimension $n$ for an unramified extension $E/F$.  Fix a self-dual lattice $\Lam_n\subset V_n$ and consider the associated group $\calo_F$-schemes, $\mathcal{G}=\mathbb{U}(\Lam_n\oplus \Lam_n)$ and $\mathcal{H}=\mathbb{U}(\Lam_n)\times\mathbb{U}(\Lam_n)$, the involution $\theta$ extends naturally to an automorphism of $\calo_F$-schemes
\[
\theta:\mathcal{G}\lra \mathcal{G}
\]
with $\mathcal{G}^\theta=\mathcal{H}$.
\begin{Lem}\label{Lem: unitary nice}
The symmetric pair $(\mathcal{G},\mathcal{H})=(\mathbb{U}(\Lam_n\oplus\Lam_n),\mathbb{U}(\Lam_n)\times\mathbb{U}(\Lam_n))$ is nice and simply connected.
\end{Lem}
\begin{proof}
Lemma \ref{Lem: descendants} shows that the base change of this pair to any field has connected reductive stabilizers, showing that the pair is simply connected. To see that the variety is nice, we appeal to our explicit construction of the categorical quotient map
\begin{align*}
    \chi:\G/\rH&\lra \A^{n}_F\\
    x=\left(\begin{array}{cc}
        A & B \\
        -B^\ast & D
    \end{array}\right)&\longmapsto (a_1(x),\ldots,a_n(x)),
\end{align*}
where $\{a_i(x)\}$ are the coefficients of the characteristic polynomial of $A$. This map is clearly defined over $\calo_F,$ and $\calo_F[a_1,\ldots,a_n]$ provides the necessary integral model $\cala=\Spec(\calo_F[a_1,\ldots,a_n])$.
\end{proof}

\section{Relative endoscopy}\label{Section: relative endo integrals}
Let  $E/F$ be a quadratic extension of $p$-adic local fields. We continue to let $W_1$ and $W_2$ denote two $n$-dimensional Hermitian vector spaces over $E$, and set $W=W_1\oplus W_2$. Set $\G=\U(W)$ and denote by $\theta:\G\to \G$, the unitary involution with $\rH=\U(W_1)\times \U(W_2)=G^{\theta}$. Let $\calq=\G/\rH$ be the associated symmetric space, and let $s:\G\lra \calq$ denote the symmetrization map.

\subsection{Relative endoscopy}\label{Section: rel end symmetric space}
Recall that $\calv_n$ denotes a fixed set of representatives of the isometry classes of Hermitian form over $E$ on $E^n$. Since $E/F$ is an extension of $p$-adic fields, $|\calv_n|=2$. If $E/F$ is unramified, we assume that the split Hermitian form is represented by $I_n$ and set $V_n:=(E^n,I_n)$ for the split Hermitian space.

Following \cite{Leslieendoscopy}, we have the following definition.
\begin{Def}
We define an \emph{(elliptic) relative endoscopic datum} of the symmetric space $\calq$ to be a triple $\Xi_{a,b}=(\xi_{a,b}, \al, \be)$, where 
\[
\xi_{a,b}=(\U(V_a)\times \U(V_b),s,\eta)
\]is an elliptic endoscopic triple for $U(W_1)$ and  $\al\in\calv_a$ (resp. $\be\in\calv_b$). Setting
\[
\calq_{a,\al}:=\U(V_a\oplus V_\al)/\U(V_a)\times \U(V_\al)\:\text{ and }\:\calq_{b,\be}:=\U(V_b\oplus V_\be)/\U(V_b)\times \U(V_\be),
\] we define the associated endoscopic symmetric space to be $\calq_{a,\al}\times \calq_{b,\be}$.
\end{Def}
For a fixed endoscopic datum, the endoscopic symmetric space is equipped with a contraction map as in Lemma \ref{Lem: unitary halfway}
\begin{align*}
R_{\al,\be}:\calq_{a,\al}\times \calq_{b,\be}&\lra \Herm(V_a)\oplus \Herm(V_b),\\
(x_a,x_b)&\longmapsto (R(x_a), R(x_b)).
\end{align*}
As in Lemma \ref{Lem: cat quotient unitary}, the coefficients of the characteristic polynomials
\[
(\chi_{x_a}(t),\chi_{x_b}(t))=(\det(tI_a-R(x_a)),\det(tI_b-R(x_b)))
\]
gives the categorical quotient of $\calq_{a,\al}\times \calq_{b,\be}$.

\begin{Def}
 Let $x\in \calq^{rss}(F)$ and $(x_a,x_b)\in \left(\calq_{a,\al}\times \calq_{b,\be}\right)^{rss}(F)$. 
We say that $x$ \emph{matches} $(x_a,x_b)$ (or that $x$ is an image of $(x_a,x_b)$) if we have the identity
\[
\chi_x(t) = \chi_{x_a}(t)\chi_{x_b}(t),
\]
where $\chi$ is the invariant map from Lemma \ref{Lem: cat quotient unitary}.
\end{Def}
In particular, the elements $x$ and $(x_a,x_b)$ match if and only if
\[
R(x)=y\in \Herm(W_1)\text{ and }R_{\al,\be}(x_a,x_b)= (y_a,y_b)\in \Herm(V_a)\oplus \Herm(V_b)
\]
match in the sense of Definition \ref{Def: endoscopic matching}. When $W_1\cong V_a\oplus V_b$, we say that $(x,(x_a,x_b))$ are a \emph{good matching pair} if $(y,(y_a,y_b))$ are.

 Given matching elements $(x,(x_a,x_b))$, we define the transfer factor by 
 \begin{equation}\label{eqn: relative transfer factors}
\Delta_{rel}((x_a,x_b),x) := \Delta((y_a,y_b),y),
\end{equation}
where the right-hand side is the Langland-Shelstad-Kottwitz transfer factor from Section \ref{Section: transfer factor}.

\subsection{Smooth transfer}

 Fix $x\in \calq^{rss}(F)$ and let $\Xi_{a,b}$ be a relative endoscopic datum. The endoscopic triple $\Xi=(\U(V_a)\times \U(V_b),s,\eta)$ of $U(W_1)$ determines a character 
 \[
 \kappa:\mathfrak{C}(\U(W_1)_{R(x)},\U(W_1);F)\lra \cc^\times
 \]
 via the construction of Lemma \ref{Lem: endo character}. Using Lemma \ref{Lem: kernel isom}, we pull this character back along the isomorphism
 \[
 \mathfrak{C}(\rH_x,\rH;F)\iso \mathfrak{C}(\U(W_1)_{R(x)},\U(W_1);F),
 \]
to obtain a character which we also call $\ka: \mathfrak{C}(\rH_x,\rH;F) \lra\cc^\times$. Using $e(T)=1$ for any torus, we thus obtain the relative $\kappa$-orbital integral
\begin{equation*}
\Orb^\kappa(x,f) := \sum_{[x']\in\calo_{st}(x)}\kappa(\inv(x,x'))\Orb(x',f).
\end{equation*}

\begin{Def}\label{Def: transfer variety}
Let $f\in C_c^\infty(\calq(F))$ and let $f^{\al,\be}\in C^\infty_c(\calq_{a,\al}(F)\times \calq_{b,\be}(F))$. We say that $f$ and $f^{\al,\be}$ are smooth transfers of each other (or match) if the following conditions are satisfied:
\begin{enumerate}
\item For any matching orbits $x\in \calq^{rss}(F)$ and $(x_a,x_b)\in \calq_{a,\al}(F)\times \calq_{b,\be}(F)$, we have an identify
\begin{equation*}
\SO((x_a,x_b),f^{\al,\be})= \Delta_{rel}((x_a,x_b),x)\Orb^\kappa(x,f).
\end{equation*}
\item If there does not exist  $x\in \calq^{rss}(F)$ matching $(x_a,x_b)\in \calq_{a,\al}(F)\times \calq_{b,\be}(F)$, then
\begin{equation*}
\SO((x_a,x_b),f^{\al,\be})= 0.
\end{equation*}
\end{enumerate}
\end{Def}
We conjecture that transfers always exist. For test functions supported in $\calq^{iso}(F)$, the existence of transfers readily follows from smooth transfer on the unitary Lie algebra.
\begin{Prop}\label{Prop: regular transfer}
Let $f\in C^\infty_c(\calq(F))$ and assume $\mathrm{supp}(f)\subset \calq(F)^{iso}$. Then there exists $f^{\al,\be}\in  C^\infty_c(\calq_{a,\al}(F)\times \calq_{b,\be}(F))$ such that $f$ and $f^{\al,\be}$ match.
\end{Prop}

\begin{proof}
The argument relies on the properties of the contraction map on $\calq^{iso}$ to reduce the statement to the existence of smooth transfer on the Hermitian Lie algebra. It is very similar to the proof of Proposition 4.5 of \cite{Leslieendoscopy}; we omit the details. 
\end{proof}

\begin{Rem}
Using the discussion in Section \ref{Section: orbital ints and reduction}, we may pull all these notions back to statements of relative $\ka$-orbital integrals on $U(W)$ and the (pure inner forms of) endoscopic groups $U(V_a\oplus V_\al)$ and $U(V_b\oplus V_\be)$. We leave these details to the interested reader.
\end{Rem}

\subsection{The fundamental lemma for the unit element}
Now assume that $E/F$ is an unramified quadratic extension of $p$-adic fields and assume that our Hermitian spaces satisfy $V_n=W_1=W_2$. In this unramified setting, we will append our groups with a subscript $n$ to differentiate by rank of the associated symmetric space; for example, $\rH_n(F)=U(V_n)\times U(V_n)$. 

Fixing a self-dual lattice $\Lam_n\subset V_n$, let $\mathrm{G}_n$ and $\mathrm{H}_n$ denote the corresponding smooth group schemes over $\calo_F$. We obtain hyperspecial subgroups 
\[
\mathrm{H}_n(\calo_F):=\U(\Lam_n)\times U(\Lam_n)\subset \rH_n(F)
\]and 
\[
\mathrm{G}_n(\calo_F):=\U(\Lam_n\oplus \Lam_n)\subset \G_n(F).
\]
Set $\bfun_{G_n(\calo_F)}$ to be the associated the associated characteristic function.

 Now consider the symmetric space $\calq_n:=\mathrm{G}_n/\mathrm{H}_n$. Recall that Corollary \ref{Cor: surjects on integral points} implies that we have a short exact sequence of pointed sets
\begin{equation*}
1\lra \mathrm{H}_n(\calo_F)\lra \mathrm{G}_n(\calo_F)\lra \calq_n(\calo_F)\lra 1.
\end{equation*}
This is compatible with the sequence (\ref{eqn: long exact sequence}) on $F$-points 
and with our normalization of measures implies the equality 
\begin{equation}\label{eqn: basic functions}
\bfun_{\calq_n(\calo_F)}=s_!(\bfun_{G_n(\calo_F)})\in C_c^\infty (\calq_n(F)).
\end{equation}
 Now suppose that $\Xi_{a,b}=(\xi_{a,b},\al,\be)$ is an elliptic relative endoscopic datum. Our measures conventions in Section \ref{measures} ensure that the given hyperspecial maximal subgroups of $\rH_n(F)$ and 
 \[
 \rH_a(F)\times \rH_b(F)
 \]each have volume $1$. We now state the main result of this article.



\begin{Thm}\label{Thm: fundamental lemma}
Assume that the characteristic of $F$ satisfies the assumption of Lemma \ref{Lem: eignvalue restrictions}. If $(\al,\be) = (I_a,I_b)$, the functions $\bfun_{\calq_n(\calo_F)}$ and $\bfun_{\calq_a(\calo_F)}\otimes\bfun_{\calq_b(\calo_F)}$ match. Otherwise, $\bfun_{\calq_n(\calo_F)}$ matches $0$.
\end{Thm}

Combining (\ref{eqn: basic functions}) with (\ref{eqn: OI first reduction}), one obtains a matching of $\ka$-orbital integrals between the test functions
\[
\bfun_{\G_n(\calo_F)}\:\text{ and }\:\bfun_{\G_a(\calo_F)}\otimes\bfun_{\G_b(\calo_F)}.
\]
Note that $\G_a\times \G_b$ is an unramified elliptic endoscopic group of $\G_n$.

\subsection{Proof of Theorem \ref{Thm: fundamental lemma}}\label{Section: proof wrap-up}

The proof of this theorem occupies Sections \ref{Section: proof of FL} and \ref{Section: descent final}. For the readers convenience, we summarize the components of the argument here.

We begin with the following simple reduction.

\begin{Lem}\label{Lem: reduce to hyperspecial}
Suppose that $(x_a,x_b)\in \calq_a(\calo_F)\times \calq_b(\calo_F)$. Then there exists $x\in \calq_n(\calo_F)$ matching $(x_a,x_b)$. In particular, if $x\in\calq_n^{rss}(F)$ is not in the same stable orbit as an element $x'\in\calq_n(\calo_F)$, then $x$ does not match any integral element of $\calq_a(F)\times \calq_b(F)$.
\end{Lem}

\begin{proof}
In accordance with our conventions on split Hermitian forms, we may fix a basis of $V_n$, $V_a$, and $V_b$ such that the forms are represented by the respective identity matrices. Writing 
\[
(x_a,x_b)=\left(\left(\begin{array}{cc}
A_1     & B_1 \\
    -{B}^\ast_1 &D_1 
\end{array}\right),\left(\begin{array}{cc}
A_2     & B_2 \\
    -{B}^\ast_2 &D_2 
\end{array}\right)\right),
\]
it is clear that 
\[
x=\left(\begin{array}{cccc}
A_1    & & B_1 &\\
&A_2&&B_2\\
    -{B}^\ast_1 &&D_1&\\ 
&-{B}^\ast_2 &&D_2 
\end{array}\right)\in \calq_n(\calo_F)
\]
matches $(x_a,x_b)$. Indeed,
\[
R_n(x) = \left(\begin{array}{cc}
    A_1 &  \\
     & A_2
\end{array}\right)=\left(\begin{array}{cc}
    R_a(x_a) &  \\
     & R_b(x_b)
\end{array}\right),
\]
showing that $x$ and $(x_a,x_b)$ are a nice matching pair.
\end{proof}

Now suppose that $\Xi_{a,b}$ is a relative endoscopic datum and suppose that $x\in \calq^{rss}(F)$ matches $(x_a,x_b)\in \calq_{a,\al}(F)\times \calq_{b,\be}(F)$. If the stable orbit of $x$ fails to meet $\calq_n(\calo_F)$, then\[\Orb^\ka(x,\bfun_{\calq_n(\calo_F)})=0,\] and the fundamental lemma at $x$ follows when $(\al,\be)$ is ramified. In the unramified case $(\al,\be)=(I_a,I_b)$, Lemma \ref{Lem: reduce to hyperspecial} forces \[\SO((x_a,x_b),\bfun_{\calq_a(\calo_F)}\otimes \bfun_{\calq_b(\calo_F)})=0,\] finishing the proof in this case.

We now assume that $x\in \calq_n(\calo_F)$. For $\nu=\pm1$, we define the $\nu$-very regular locus by
\[
\calq_n^{\heartsuit,\nu}(\calo_F)=\{x\in\calq_n^{rss}(\calo_F): \overline{x}\in \calq_n(k)-D_{\nu}(k)\},
\]
where 
\[
D_{\nu}(k)=\{X\in \End(\Lam_n/\vp\Lam_n): \det(\nu I_n-X)=0\}.
\]
If $x\in \calq_n^{\heartsuit,\nu}(\calo_F)$, the fundamental lemma at $x$ is shown in Proposition \ref{Prop: very reg FL}. If $x\notin \calq_n^{\heartsuit,\nu}(\calo_F)$ for either $\nu=\pm1,$ we must apply the descent techniques developed in Sections \ref{Section: TJD and descent} and \ref{Section: relative kazhdan}. The fundamental lemma at $x$ is shown in Proposition \ref{Prop: final cases}, which follows from the descent formulas in Lemmas \ref{eqn: almost there general}, \ref{eqn: stable almost there general}, and \ref{eqn: almost there transfer factor}. This completes the proof of Theorem \ref{Thm: fundamental lemma}.




\section{The infinitesimal theory and the very regular locus}\label{Section: proof of FL}

We begin by recalling the infinitesimal fundamental lemma from \cite{LeslieUFJFL}. We then study the Cayley transform. Through this quasi-exponential map, we reduce Theorem \ref{Thm: fundamental lemma} to the Lie algebra case over the very regular locus. In Section \ref{Section: descent final}, we apply the descent methods of Section \ref{Section: TJD and descent} to complete the proof.

\subsection{The Lie algebra of the symmetric space}
Consider the Lie algebra $\fu(W)$ of $\G=\U(W)$. The differential of the involution $\theta$ acts on $\fu(W)$ by the same action and induces a $\zz/2\zz$-grading
\[
\fu(W)= \fu(W)_0\oplus \fu(W)_1,
\]
where $\fu(W)_i$ is the $(-1)^i$-eigenspace of the map $d\theta$. Then the pair $$(\rH,\fu(W)_1)$$ is called the infinitesimal symmetric space associated to $\calq$. This means that if $x_0\in \calq_n(\calo_F)\subset\calq_n(F)$ denotes the distinguished $\rH(F)$-fixed point, then
\[
T_{x_0}(\calq_n)(F) \cong\fu(W)_1,
\]
with $\rH(F)$ acting by restriction of the adjoint representation. We have a natural isomorphism of $H(F)$-representations
\begin{align*}
    \fu(W)_1&\cong \Hom_E(W_2,W_1)\\
    \de(X)=\left(\begin{array}{cc}
     & X \\
    -X^\ast & 
\end{array}\right)&\longmapsto X,
\end{align*}
where the action on the right-hand side is given by pre- and post-composition; we frequently identify $\fu(W)_1$ and $\Hom_E(W_2,W_1)$ via this map in the sequel. 

We also recall the (infinitesimal) contraction map 
\begin{align*}
    r:\fu(W)_1&\lra \Herm(W_1),\\
    \de(X)&\longmapsto-XX^\ast.
\end{align*} This gives a categorical quotient of the $U(W_2)$-action on $\fu(W)_1$ \cite[Lemma 3.2]{Leslieendoscopy}. Taking the coefficients of the characteristic polynomial $\car_{r(\de)}(t)$ gives the categorical quotient for the $\rH$-action.

\subsection{Relative endoscopy for the Lie algebra}\label{Section: relative endoscopy}
We briefly recall the notions of relative endoscopy from \cite{Leslieendoscopy}. Fix a elliptic relative endoscopic datum $\Xi_{a,b}=(\xi_{a,b},\al,\be)$. As before, we denote $V_\al = (E^a,\al)$ and $V_\be=(E^b,\be)$ and consider the Lie algebras
\[
\fu(V_a\oplus V_\al)\text{   and    }\fu(V_b\oplus V_\be),
\]
and associated symmetric pairs
\[
\left(\U(V_a)\times \U(V_\al),\fu(V_a\oplus V_\al)_1\right) \text{ and }\left(\U(V_b)\times \U(V_\be),\fu(V_b\oplus V_\be)_1\right).
\]
In \cite{Leslieendoscopy}, we define the direct sum of these symmetric pairs to be an infinitesimal {endoscopic symmetric pair} associated to $\Xi_{a,b}$. It is clear that it is the tangent space at the identity of the endoscopic symmetric space associated to $\Xi_{a,b}$ defined in Section \ref{Section: relative endoscopy}. In particular, the theory we develop here is compatible with that of \cite{Leslieendoscopy}.

This space comes equipped with the contraction map (see \cite[Section 3]{Leslieendoscopy} for details)
\begin{align*}
r_{\al,\be}:\fu(V_a\oplus V_\al)_1\oplus \fu(V_b\oplus V_\be)_1&\lra \Herm(V_a)\oplus \Herm(V_b),\\
(\de_a,\de_b)&\longmapsto (r(\de_a),r(\de_b)).
\end{align*}
We say that a regular semi-simple element $\de\in \fu(W)_1^{rss}$  \textbf{matches} the pair $$(\de_a,\de_b)\in [\fu(V_a\oplus V_\al)_1\oplus\fu(V_b\oplus V_\be)_1]^{rss}$$ if
\[
\car_{r(\de)}(t)=\car_{r(\de_a)}(t)\car_{r(\de_b)}(t).
\]
In particular, $r(\de)\in \Herm(W_1)$ and $r_{\al,\be}(\de_a,\de_b)\in \Herm(V_a)\oplus\Herm(V_b)$ match in the sense of Definition \ref{Def: endoscopic matching}; we similarly define when $(\de, (\de_a,\de_b))$ is a good matching pair. 
 For matching elements $(\de_a,\de_b)$ and $\de$, we define the transfer factor
\begin{equation*}
\tilde{\De}_{rel}((\de_a,\de_b),\de):=\De(r_{\al,\be}(\de_a,\de_b), r(\de)),
\end{equation*}
where the right-hand side is the Langlands-Shelstad-Kottwitz transfer factor the twisted Lie algebra. The notion of smooth transfer of orbital integrals was studied in \cite{Leslieendoscopy}; this again is mirrored in Section \ref{Section: relative endo integrals} so we omit the details.
\subsection{The infinitesimal fundamental lemma}
For the remainder of the paper, we assume that $E/F$ is an unramified extension of non-archimedean local fields of characteristic zero. Suppose that $V_n=W_1=W_2$ is split, and let $\Lam_n\subset V_n$ be a self-dual lattice. In this case,
\[
\fu(W)_1= \Hom_E(V_n,V_n) =\End(V_n)
\] 
and the ring of endomorphisms $\End(\Lam_n)\subset \End(V_n)$ of the lattice $\Lam_n$ is a compact open subset. The following was proved in \cite{LeslieUFJFL}.
\begin{Thm}\label{Thm: fundamental lemma Lie alg}
 If $(\al,\be) = (I_a,I_b)$, the functions $\bfun_{\End(\Lam_n)}$ and $\bfun_{\End(\Lam_a)}\otimes\bfun_{\End(\Lam_b)}$ match. Otherwise, $\bfun_{\End(\Lam_n)}$ matches $0$.
\end{Thm}
Our goal is to show that this result implies Theorem \ref{Thm: fundamental lemma}.

\subsection{The Cayley transform}\label{Section: cayley}
For any $\xi\in E$, we define
\[
D_\xi=\{X\in \End(W): \det(\xi I-X)=0\}.
\]
\begin{Lem}\label{Lem: Cayley map}
 The Cayley transform
\begin{align*}
\fc_{\pm1}:\End(W)-D_1(F)&\lra \GL(W)\\
			X&\longmapsto \mp(1+X)(1-X)^{-1}
\end{align*}
induces a $U(V_1)\times U(V_2)$-equivariant isomorphism between $\fu(W)_1-D_1(F)$ and $\calq(F)-D_{\pm1}(F)$. The images of $\fu(W)_1-D_1(F)$ under $\fc_{\pm}$ form a finite cover by open subsets of $\calq(F)-(D_1\cap D_{-1})(F)$. In particular, the images contain the regular semi-simple locus of $\calq(F)$.
\end{Lem}
\begin{proof}
It is well known \cite[Lemma 3.4]{ZhangFourier} that for any $\nu\in E$ the Cayley map 
\[
\fc_\nu(X) = -\nu(1+X)(1-X)^{-1}
\]induces a $\GL(W)$-equivariant isomorphism between $\fgl(W)-D_1(F)$ and $\GL(W)-D_\nu(F)$. Indeed, the inverse is of the same form: for $x\in \GL(W)-D_\nu(F)$, set
\[
\be_\nu(x) = -(\nu+x)(\nu-x)^{-1}.
\]
This gives the required inverse transformation. It is easy to check that the constraint that $\fc_\nu(X)\in U(W)$ whenever $X\in \fu(W)$ forces $\nu\overline{\nu}=1$. 

We now show that for the transform to compatible with the involution $\theta$, we need $\nu=\pm1$. Indeed, if $\theta(X)=\ep X \ep =-X$, then 
\[
\theta(\fc_\nu(X)) = -\nu(1-X)(1+X)^{-1}= ({-\nu}^{-1}(1+X)(1-X)^{-1})^{-1}=\fc_{\nu^{-1}}(X)^{-1}.
\]
Thus, $\theta(\fc_\nu(X))=\fc_\nu(X)^{-1}$ if and only if $\nu=\pm1$. Furthermore, recalling the symmetrization map $s(g)=g\theta(g)^{-1}$, we see that
\[
\fc_{\pm1}(X) = \mp s(1+X)
\] whenever $X\in \fu(W)_1-D_1(F)$. Lemma \ref{Lem: minus sign} implies that both $\mp s(1+X)\in\calq(F)$ and we conclude that we obtain a pair of morphisms
\[
\fc_{\pm1}: \textbf{$\fu(W)_1-D_1(F)\lra\calq(F)-D_{\pm1}(F)$}.
\]
For $\nu=\pm1,$ suppose now that $x\in \calq(F)-D_{\pm1}(F)$, so that $\theta(x) = \ep x\ep =x^{-1}$. Then 
\begin{align*}
  \theta(\be_\nu(x))&=-(\nu+x^{-1})(\nu-x^{-1})^{-1}\\
                    &=-(\nu x+I)(\nu x-I)^{-1}\\ &=-(\nu +x)(-\nu+ x)^{-1}=- \be_\nu(x),
\end{align*}
where we made use in the third equality of the fact that $\nu=\nu^{-1}$. This implies that $c_\pm$ induces a $\GL(W)$-equivariant isomorphism between $\fgl(W)-D_1(F)$ and $\GL(W)-D_\nu(F)$.

For the final statement, we can check over $\Fbar$. Using (\ref{Prop: GL orbit reps}), we see that any $x\in \calq^{rss}(F)$ lies in the same $\rH(\Fbar)$-orbit as an element of the form $x(A,0,0)$ where $A\in \fgl(W_1)$ is regular semi-simple with eigenvalues avoiding $\pm1$. Considering Lemma \ref{Lem: characteristic comparison}, the same is true of roots of $\car_x(t),$  implying that
\[
\calq^{rss}\subset \calq-D_1\cup D_{-1}.\qedhere
\]
\end{proof}



The next few lemmas show how the Cayley transform is compatible with the categorical quotients considered in Section \ref{Section: unitary orbits}, our notions of matching, and transfer factors.

\begin{Lem}\label{Lem: Cayley contraction}
For $\nu=\pm1$, there is a commutative diagram
\[
\begin{tikzcd}
\fu(W)_1-D_1(F)\ar[r,"\fc_{\nu}"]\ar[d,"r"]& {\calq}(F)-D_{\nu}(F)\ar[d,"R"]\\
\Herm(W_1)-D_1(F)\ar[r,"\fc_\nu"]&\Herm(W_1)-D_{\nu}(F),
\end{tikzcd}
\]
where by abuse of notation we let $\fc_\nu$ denote the Cayley transform on both $\fu(W)$ and $\Herm(W_1)$.
\end{Lem}
\begin{proof}
Let $\de\in \fu(W)_1-D_1(F)$, and let $X\in \Hom_E(W_2,W_1)$ such that
\[
\de = \left(\begin{array}{cc}
     & X \\
    -X^\ast & 
\end{array}\right),\text{ so that  } r(\de)=-XX^\ast.
\]
We need to calculate
\[
\fc_\nu(\de) = -\nu\left(\begin{array}{cc}
    I & -X \\
    X^\ast & I
\end{array}\right)\left(\begin{array}{cc}
    I & X \\
    -X^\ast & I
\end{array}\right)^{-1}.
\]
A simple matrix computation shows that
\[
\left(\begin{array}{cc}
    I & X \\
    -X^\ast & I
\end{array}\right)^{-1}= \left(\begin{array}{cc}
    (I+XX^\ast)^{-1} & -X (I+X^\ast X)^{-1} \\
    X^\ast (I+XX^\ast)^{-1}&  (I+X^\ast X)^{-1}
\end{array}\right),
\]
In particular, we have
\[
\fc_{\nu}(\de) =  -\nu\left(\begin{array}{cc}(I-XX^\ast)(I+X X^\ast)^{-1}&-2X(I+X^\ast X)^{-1}\\2X^\ast(I+XX^\ast)^{-1}&(I-X^\ast X)(I+X^\ast X)^{-1}\end{array}\right). 
\]
 The commutativity of the diagram now follows from the definitions of $r$ and $R$.
\end{proof}

We now consider the effect of the Cayley transform on the invariant polynomial maps.

\begin{Lem}\label{Lem: Cayley cahracteristic}
Let $\de\in \End(W)-D_1(F)$ and set $x=\fc_\nu(\de)$. Then
\begin{equation*}
\car_x(t) = (t-\nu)^{\dim(W)}\car_\de(1)^{-1}\left(\car_\de\left(\frac{t+\nu}{t-\nu}\right)\right).
\end{equation*}
\end{Lem}
\begin{proof}
We may assume that $F=\Fbar$. It is a straightforward exercise that if $\lam\in \Fbar^\times$ is an eigenvalue of multiplicity $m(\lam)$ of $\de$, then 
\[
\lam':=-\nu\left(\frac{1+\lam}{1-\lam}\right)
\]is an eigenvalue of $x$ with the same multiplicity. The rational function
\[
\car_\de\left(\frac{t+\nu}{t-\nu}\right)
\]
then vanishes on the eigenvalues of $x$. In particular, 
\[
(t-\nu)^{\dim(W)}\car_\de(1)^{-1}\left(\car_\de\left(\frac{t+\nu}{t-\nu}\right)\right)
\]
is a monic polynomial with the correct roots and multiplicities, and so must be $\car_x(t).$
\end{proof}

Before we apply this to comparing transfer factors, we check that the Cayley transform preserves our notions of matching orbits.

\begin{Lem}\label{Lem: cayley matching}
Suppose that $\de\in \fu(W)^{rss}_1-D_{\pm}$, and let $x=\fc_{\nu}(\de)\in \calq^{rss}(F)$. Fix an elliptic datum $\Xi_{a,b}$. 
Then 
\[(\de_a,\de_b)\in \fu(V_a\oplus _\al)_1\oplus \fu(V_b\oplus V_\be)_1
\]matches $\de$ if and only if 
\[
(x_a,x_b)= (\fc_{\nu}(\de_a),\fc_{\nu}(\de_b))\in \calq_{a,\al}(F)\times \calq_{b,\be}(F)
\]matches $x$. 
\end{Lem}
\begin{proof}
This follows immediately from our definition of matching and the preceding  characteristic polynomial calculation. Indeed, if $x=\fc_\nu(\de)$, then Lemma \ref{Lem: Cayley contraction} implies that $R(x) = \fc_{\nu}(r(\de))$. Now Lemma \ref{Lem: Cayley cahracteristic} implies there is a bijection
\[
\lam\mapsto -\nu\left(\frac{1+\lam}{1-\lam}\right)
\] between roots of the characteristic polynomials of $R(x)$ and $r(\de)$. This implies that the invariant polynomial of $x$ has the same roots as that of $(x_a,x_b)$ if and only if the same holds for $\de$ and $(\de_a,\de_b)$.
\end{proof}

\begin{Lem}\label{Lem: Cayley disc}
With notation as in the previous lemma with $x=\fc_\nu(\de)$, consider the relative discriminants
\[
\tilde{D}_{a,b}(\de) = \prod_{t_a,t_b}(t_a-t_b)\quad\text{ and }\quad {D}_{a,b}(x) = \prod_{z_a,z_b}(z_a-z_b),
\]
where $t_a$ (resp. $z_a$) runs over the roots of the invariant of $\de$ (resp. $x$) arising from $\de_a$ (resp. $y_a$) and likewise with $t_b$ (resp. $y_b$), then
\[
\tilde{D}_{a,b}(\de) =(-2\nu)^{a\cdot b}\prod_{z_a,z_b}\frac{1}{(z_a-\nu)(z_b-\nu)}D_{a,b}(x).
\]

\end{Lem}
\begin{proof}
Lemma \ref{Lem: Cayley cahracteristic} implies that if $\tilde{\chi}_{\de}(t)$ is the invariant of $\de$, then
\[
\chi_x(t)=(t-\nu)^n\tilde{\chi}_\de(1)^{-1}\left(\tilde{\chi}_\de\left(\frac{t+\nu}{t-\nu}\right)\right)
\]
is the invariant of $x$. It is clear that $\pm1$ is not a root of this polynomial, and we see that if $z$ is a root of $\chi_x(t)$, then $t=\frac{z+\nu}{z-\nu}$ is a root of $\tilde{\chi}_\de(t)$ of the same multiplicity.

Applying this to the discriminants, we have
\begin{align*}
\tilde{D}_{a,b}(\de) &= \prod_{t_a,t_b}(t_a-t_b)\\
				&=\prod_{z_a,z_b}\left(\frac{z_a+\nu}{z_a-\nu}-\frac{z_b+\nu}{z_b-\nu}\right)\\
				&=\prod_{z_a,z_b}\frac{-2\nu(z_a-z_b)}{(z_a-\nu)(z_b-\nu)}.
\end{align*}
Counting the number of factors gives the coefficient $(-2\nu)^{a\cdot b}$.
\end{proof}


\subsection{The $\nu$-very regular locus}\label{Section: very regular} 
Let $\nu=\pm1$. If the pairs $(x,(x_a,x_b))$ and $(\de,(\de_a,\de_b))$ are as in Lemmas \ref{Lem: cayley matching} and \ref{Lem: Cayley disc}, we set
\[
C_{a,b,\nu}(x,\de):=(-2\nu)^{a\cdot b}\prod_{z_a,z_b}\frac{1}{(z_a-\nu)(z_b-\nu)}.
\]
The results of the previous section imply the formula
\begin{equation*}
    \tilde{\De}_{rel}((\de_a,\de_b),\de) = \eta_{E/F}(C_{a,b,\nu}(x,\de))|C_{a,b,\nu}(x,\de)|_F\De_{rel}((x_a,x_b),x).
\end{equation*}
Since $E/F$ is unramified and the residue characteristic is odd, these transfer factors agree whenever $C_{a,b,\nu}(x,\de)$ is a unit. In this section, we define certain open subsets of $\calq_n(\calo_F)$ for which this is the case.

Identifying $\fu(\Lam_n\oplus \Lam_n)_1=\End(\Lam_n)$, we define the \textbf{very regular locus} of $\End(\Lam_n)$ 
\[
\End(\Lam_n)^\heartsuit=\{\de\in \End(\Lam_n)^{rss}-D_{1}(\calo_F): \overline{\de}\in \End(\Lam/\vp\Lam)-D_{1}(k) \},
\]
where $\overline{\de}\in\End(\Lam/\vp\Lam)$ denotes the image of $\de$ under the modular reduction map. Similarly, we define  the $\nu$-\textbf{very regular locus} of $\calq_n(\calo_F)$
\[
\calq_n^{\heartsuit,\nu}(\calo_F)=\{x\in\calq_n^{rss}(\calo_F): \overline{x}\in \calq_n(k)-D_{\nu}(k)\}.
\]
\begin{Lem}\label{Lem: very reg count}
Suppose that $x\in \calq_n(F)$ and suppose $x=\fc_\nu(\de)$ for $\de\in \fu(W)_1$. Then $x\in \calq_n^{\heartsuit,\nu}(\calo_F)$ if and only if $\de\in\End(\Lam_n)^{\heartsuit}$.
\end{Lem}
\begin{proof}
We begin by assuming $x\in \calq_n^{\heartsuit,\nu}(\calo_F)$. This implies that if $\lam\in\calo_{\Fbar}^\times$ is a root of the characteristic polynomial of $x$, then $\nu -\lam\in\calo_{\Fbar}^\times$. In particular,  $\nu-x\in \GL(\Lam_n)$. Clearly, we also have $\nu+x\in \End(\Lam_n)$, implying that 
\[
\de=\fc_\nu^{-1}(x) = -(\nu+x)(\nu-x)^{-1}\in \End(\Lam_n).
\]
In particular, $\overline{\de}=\fc_\nu^{-1}(\overline{x})$ and Lemma \ref{Lem: Cayley cahracteristic} implies that $1$ is not a root of $\car_{\overline{\de}}(t)$.

Conversely, assume $\de\in \End(\Lam_n)^\heartsuit$ with $x=\fc_\nu(\de)$. Similar elementary considerations imply that $1-\de\in \GL(\Lam_n)$, so that
\[
x=-\nu(1+\de)(1-\de)^{-1}\in \calq_n^{\heartsuit,\nu}(\calo_F).\qedhere
\]
\end{proof}
This shows that
\[
\calq_n^{\heartsuit,\nu}(\calo_F)=\fc_\nu\left(\End(\Lam_n)^\heartsuit\right).
\]
In particular, for any $x\in \calq_n^{\heartsuit,\nu}(\calo_F)$, the reduction $\overline{x}\in\calq_n(k)$ is in the image of the Cayley transform. 

The following lemma is immediate from the definitions.
\begin{Lem}
Let $x\in \calq_n^{\heartsuit,\nu}(\calo_F)$ and suppose $x'\in\calq_n(\calo_F)$ lies in the stable orbit of $x$. Then $x'\in\calq_n^{\heartsuit,\nu}(\calo_F)$
\end{Lem}
\begin{proof}
We need to show that $\overline{x'}\notin D_\nu(k)$. But $D_\nu(k)$ is closed under the stable action of $\mathcal{H}(\kbar)$ and $\overline{x}\notin D_\nu(k)$.
\end{proof}

Define $\calq_n^{\heartsuit, \nu}(F)$ to be the open subset of $\calq_n^{rss}(F)$ consisting of elements in the same stable orbit as an element in $\calq_n^{\heartsuit,\nu}(\calo_F)$; we similarly define $\End(V_n)^{\heartsuit}$. This locus may be characterized as those elements of $\calq_n(F)$ with eigenvalues $\lam\in \calo^\times_{F^{alg}}$ all satisfying 
\[
|\lam-\nu|=1.
\]
The next proposition shows that the fundamental lemma holds for these open sets of $\calq_n(\calo_F)$.

\begin{Prop}\label{Prop: very reg FL}
Fix an elliptic relative endoscopic datum $\Xi_{a,b}$, and suppose that $x\in\calq_n^{\heartsuit,\nu}(F)$. Then the fundamental lemma holds at $x$. That is, if $\ka$ is the character associated to the endoscopic datum and if $x$ and $(x_a,x_b)\in \calq_{a,\al}(F)\times\calq_{b,\be}(F)$ match, we have
\[
\De_{rel}((x_a,x_b),x)\Orb^\ka(x,\bfun_{\calq_n(\calo_F)})=\begin{cases}\SO((x_a,x_b),\bfun_{\calq_a(\calo_F)}\otimes \bfun_{\calq_b(\calo_F)})&:(\al,\be)= (I_a,I_b),\\\qquad\qquad\qquad 0&:(\al,\be)\neq (I_a,I_b).\end{cases}
\]
\end{Prop}
\begin{proof}

We begin with the transfer factor. First assume that $x$ and $(x_a,x_b)$ are a nice matching pair. Then we have an embedding 
\[
\phi_{a,b}:\Herm(V_a)\oplus \Herm(V_b)\hra \Herm(V)
\] satisfying that 
\[
\phi_{a,b}(y_a,y_b)) = y,
\]
where $R(x)= y$ and similarly for $(y_a,y_b)$ and 
\[
\De_{rel}((x_a,x_b),x) = \eta_{E/F}(D_{a,b}(x))|D_{a,b}(x)|_F.
\]
The assumption $x\in \calq_n^{\heartsuit,\nu}(F)$ implies that $x=\fc_\nu(\de)$ for some $\de\in \End(V_n)^{rss}$. Moreover, the same holds for the matching pair $ (x_a,x_b)=(\fc_\nu(\de_a),\fc_\nu(\de_b))$. Lemma \ref{Lem: cayley matching} implies that $x$ and $(x_a,x_b)$ are a good match if and only if $\de$ and $(\de_a,\de_b)$ are. 

Combining the calculation of Lemma \ref{Lem: Cayley disc} with the fact that  $C_{a,b,\nu}(\de,x)$ is a unit when $x\in \calq_n^{\heartsuit,\nu}(F)$, we have
\begin{equation}\label{eqn: transfer factor descent}
\De_{rel}((x_a,x_b),x)=\tilde{\De}_{rel}((\de_a,\de_b),\de).    
\end{equation}

In general, suppose that $x$ and $x'$ are in the same stable orbit with $\inv(x,x')\in H^1(F,\rH_x)$ the corresponding invariant. Writing $x'=\fc_\nu(\de')$, the equivariance of $\fc_\nu$ implies that $\de'$ is in the same stable class as $\de$ and that under the induced isomorphism $\rH_x\iso \rH_\de$, we have the identification 
\[
\inv(\de,\de')=\inv(x,x').
\]
Thus, the identity (\ref{eqn: transfer factor descent}) holds for any matching pair $x=\fc_\nu(\de)$ and $(x_a,x_b)=(\fc_\nu(\de_a),\fc_\nu(\de_b))$. 

We now note that Lemma \ref{Lem: very reg count} implies immediately that for any $x=\fc_\nu(\de)\in \calq_n^{\heartsuit,\nu}(F)$, 
\begin{equation*}
    \Orb(x,\bfun_{\calq_n(\calo_F)})= \Orb(\de,\bfun_{\End(\Lam_n)}).
\end{equation*} 
This shows that
\begin{equation}\label{eqn: reduce to Lie 1}
\De_{rel}((x_a,x_b),x)\Orb^\ka(x,\bfun_{\calq_n(\calo_F)})=\tilde{\De}_{rel}((\de_a,\de_b),\de)\Orb^\ka(\de,\bfun_{\End(\Lam_n)}).
\end{equation}
If $(\al,\be)\neq (I_a,I_b)$, the proposition follows from the corresponding vanishing of orbital integrals in Theorem \ref{Thm: fundamental lemma Lie alg}.

Assuming now that $(\al,\be)= (I_a,I_b)$, we further claim that
\begin{equation}\label{eqn: reduce to Lie 2}
\SO((x_a,x_b),\bfun_{\calq_a(\calo_F)}\otimes \bfun_{\calq_b(\calo_F)})=\SO((\de_a,\de_b),\bfun_{\End(\Lam_a)}\otimes \bfun_{\End(\Lam_b)}).
\end{equation}
where $(\de_a,\de_b)\in \End(V_\al)\oplus\End(V_\be)$. To see this, suppose $(x_a',x_b')\in \calq_{a}(\calo_F)\times\calq_{b}(\calo_F)$ lies in the stable orbit of $(x_a,x_b)$. Since $x\in \calq_n^{\heartsuit,\nu}(F)$, if $\lam$ is a root of $\car_x(t)$, then 
\[
\lam\in\calo_{\Fbar}^\times\:\text{ and }\:\overline{\lam}\neq \nu\in \kbar.
\]By the definition of matching of orbits and Lemma \ref{Lem: characteristic comparison}, the same is true of the roots of the characteristic polynomials of $x_a'$ and $x_b'$. In particular,
\[
(x_a',x_b')\in \calq^{\heartsuit,\nu}_a(\calo_F)\times \calq^{\heartsuit,\nu}_b(\calo_F).
\]
The equality (\ref{eqn: reduce to Lie 2}) now follows from Lemma \ref{Lem: very reg count}. This proves the proposition by combining (\ref{eqn: reduce to Lie 1}) and (\ref{eqn: reduce to Lie 2}) with the matching of orbital integrals in Theorem \ref{Thm: fundamental lemma Lie alg}.
\end{proof}

\begin{Rem}\label{Rem: last cases}
Combining Lemma \ref{Lem: reduce to hyperspecial} and Proposition \ref{Prop: very reg FL}, the fundamental lemma at $x$ is now established unless $x\in\calq_n^{rss}(\calo_F)$ has the property that 
\[
\overline{x}\in (D_1\cap D_{-1})(k).
\]
\end{Rem} 

\section{Descent to the very regular case}\label{Section: descent final}

We now use the results of Section \ref{Section: top jordan decomp rel} and Proposition \ref{Prop: relative Kazdhan lemma} to prove the final cases of Theorem \ref{Thm: fundamental lemma}. This is stated as Proposition \ref{Prop: final cases} below.

 
Suppose now that $x\in \calq_n^{rss}(\calo_F)$ such that 
\[
\overline{x}\in (D_1\cap D_{-1})(k),
\] and let $x=x_{as}x_{tu}$ be the topological Jordan decomposition. 
If $W=V_0\oplus V_1\oplus V_{-1}$ is the eigenspace decomposition of $W$ for $x_{as}$ discussed in Lemma \ref{Lem: descendants}, 
we have
\begin{align}\label{central product}
  \G_{x_{as}}= \G_{x_{as}}'\times \U(V_1)\times \U(V_{-1}),  
\end{align}
where $\G_{x_{as}}'$ is a product of groups in cases (\ref{case3})-(\ref{case5}) in Lemma \ref{Lem: descendants}.  
 
 We need the following characterization of the eigenvalues of $x_{as}$ that can occur in $V_0$.

\begin{Lem}\label{Lem: eignvalue restrictions}
Let $e$ be the ramification degree of $F/\qq_p$ and assume that $p>\max\{e+1,2\}$. Let $\nu\in \{\pm1\}$. Suppose that $x=x_{as}x_{tu}\in\calq_n^{rss}(\calo_F)$ is the topological Jordan decomposition. If $\lam\in \calo_{\Fbar}^\times$ is an eigenvalue of $x_{as}$ such that $\overline{\lam}=\nu\in \kbar$, then $\lam=\nu$.
\end{Lem}

\begin{proof}
Recall the construction of $x_{as}$: for $c_G$ as in Section \ref{Section: top jordan decomp}, and $l\in \zz_{>0}$ such that $q^l\equiv 1\pmod{c_G}$, we have $x_{as} = \lim_{m\to \infty}x^{q^{lm}}$. Thus, the eigenvalues of $x_{as}$ are 
\[
\left\{\lim_{m\to \infty}\lam^{q^{lm}}: \car_x(\lam)=0\right\}.
\]

Suppose now that $\lam_{as}\in \calo_{\Fbar}^\times$ is as in the statement of the Lemma. 
As above, there exists $\lam\in \Omega(x)$ (see Lemma \ref{Lem: characteristic comparison}) such that 
\begin{equation}\label{eqn:limit eigenvalue}
\lam_{as}=\lim_{m\to \infty}\lam^{q^{lm}}.
\end{equation}

We claim that $\overline{\lam}=\nu\in \kbar$. Indeed, since $x_{as}x_{tu}=x_{tu}x_{as}$ we may simultaneously diagonalize these semi-simple elements of $\GL(W)$ over $\Fbar$. A fortiori, this diagonalizes $x=x_{as}x_{tu}$ and shows that $\lam$ decomposes
\[
\lam=\lam_{as}\lam_{tu}
\]
where $\lam_{as}$ (resp. $\lam_{tu}$) is the associated eigenvalue of $x_{as}$ (resp. $x_{tu}$). Since $x_{tu}$ is topologically unipotent, we see that $\overline{\lam}_{tu}=1$ so that
\[
\overline{\lam}=\overline{\lam}_{as}\overline{\lam}_{tu}=\overline{\lam}_{as}\in \kbar.
\]
By assumption, we see $\overline{\lam}=\nu$.

We now claim the limit (\ref{eqn:limit eigenvalue}) is $\nu$. Let $q=p^f$ and note that $\la\vp^e\ra=\la p\ra$. We may write $\lam = \nu+V$ with $V\in \calo_{\Fbar}$ such that $v:=\val(V)>0$. For $m\geq1$, set $D=q^{lm}$. We have
\[
\lam^{D} = \nu+\sum_{j=1}^{D}{\binom{D}{j}} \nu^{(D-j)}V^j.
\]
By our assumption that $p>e+1$ and a simple argument on the divisibility of binomial coefficients (see \cite[Lemma 3.1]{Halesunramified}), we have 
\[
\val\left({\binom{D}{j}} \nu^{(D-j)}V^j\right)\geq \val\left(DV\right)=(lef)m+v.
\]
Thus, 
\[
\big|\lam^{q^{lm}}-\nu\big|_F\lra 0 \text{   as   } m\lra \infty.\qedhere
\]
\end{proof}

In particular, for  $x\in \calq_n^{rss}(\calo_F)$  such that 
\[
\overline{x}\in (D_1\cap D_{-1})(k),
\] this lemma implies that $\dim(V_1)>0$ and $\dim(V_{-1})>0$.

\subsection{Descent on $\calq_n$}\label{Section: descent1}
While Proposition \ref{Prop: relative Kazdhan lemma} allows us to descend all the way to the topologically unipotent locus of $\calq_{x_{as}}(F)$, this is not necessary in our present case. Indeed, Proposition \ref{Prop: very reg FL} implies we need only descend to the very regular locus, which is larger than the topological unipotent locus. To avoid a tedious comparison of our transfer factors to those associated to the descendants of the forms (\ref{descendants1}) and (\ref{descendants2}), we establish a slight generalization of this proposition (see Lemma \ref{Lem: into the reg locus} below).

Corollary \ref{Cor: absolute nice lifts} tells us that there exists $g\in \G_{x_{as}}(\calo_F)$ such that
\[
s(g)=x_{as}.
\] Using the decomposition (\ref{central product}), we write $g= (g',g_{1},g_{-1})\in \G_{x_{as}}(\calo_F)$, so that
\[
s(g) = (s(g'),s(g_1),s(g_{-1}))= (x_{as}|_{V_0},I_{V_{1}},-I_{V_{-1}})=x_{as}.
\]We now decompose the absolutely semi-simple part $x_{as}=\ga\cdot y_{as}$
where  
\[
\ga = (I_{V_0}, I_{V_1}, -I_{V_{-1}})\in \G_{x_{as}}(\calo_F),
\]
 and 
 \[
 y_{as}=(x_{as}|_{V_0}, I_{V_1}, I_{V_{-1}})\in \G_{x_{as}}(\calo_F).
 \]
 In particular,
 \[
 \G_\ga=\U(V_0\oplus V_1)\times \U(V_{-1})\subset \U(W),
 \]
 and $\G_{x_{as}}\subset \G_\ga$ is a twisted Levi subgroup of $\G_\ga$.  
  It is clear that both $\ga, y_{as}\in \calq_n(\calo_F)$. If we set $y=\ga^{-1}x,$ then $y=y_{as}x_{tu}\in \G_{\ga}(\calo_F)$ is the topological Jordan decomposition of $y$. Lemma \ref{Lem: unipotent image} and Proposition \ref{Prop: relative top decomp} also imply that $y\in \calq_n(F)$.
 \begin{Lem}\label{Lem: new decomp}
  With the notation as above, $y\in \calq_{\ga}^{rss}(F)$.
 \end{Lem}
 \begin{proof}
 The argument is similar to the proof of Lemma \ref{Lem: top nil reg rel}. In our present setting, $n=\frac{1}{2}\dim(W)$ gives the dimension of the centralizer of an element of $\calq_n^{rss}(F)$. As $\calq_\ga$ is a product of two lower rank analogues of $\calq_n$, it follows that the dimension of a regular stabilizer in $\calq_\ga$ is
 \[
m_\ga = \frac{1}{2}\dim(V_0\oplus V_1)+\frac{1}{2}\dim(V_{-1}) = n,
 \]
 as well (here we use the fact that $\dim(V_{-1})$ is even). Arguing as in the proof of Lemma \ref{Lem: top nil reg rel}, for $x=\ga y\in \calq^{rss}(F)$ as above,
\[
(\rH_{\ga})_{y}=\rH_{x}.
\]
This proves the claim.
 \end{proof}

Decomposing $W$ with respect to the action of $\ga$, we make a slight abuse of notation\footnote{The notation $W=W_1\oplus W_2$ was used for the eigenspace decomposition of $\ep$, but this should cause no confusion.} and write
 \[
W=W_1\oplus W_{-1},
 \]
 where $W_1=V_0\oplus V_1$ and $W_{-1}=V_{-1}$. The descendant $(\G_\ga,\rH_\ga)$ decomposes as a product
 \[
 (\G_\ga,\rH_\ga) = (\G_{\ga,1},\rH_{\ga,1})\times  (\G_{\ga,-1},\rH_{\ga,-1}) 
 \]
 where for both $\nu=\pm1$,
 \[
(\G_{\ga,\nu},\rH_{\ga,\nu}) =(\U(W_{\nu}), \U(W_{\nu,1})\times \U(W_{\nu,-1})).
 \]
Here $W_{\nu,\nu'}=\{w\in W: \ga w=\nu w,\:\: \ep w=\nu' w\}$. We have the associated symmetric spaces
 \[
 \calq_{1}=\U(W_{1})/ \U(W_{1,1})\times \U(W_{1,-1})
 \]
 and
 \[
 \calq_{-1}=\U(W_{-1})/ \U(W_{-1,1})\times \U(W_{-1,-1}).
 \]
 Proposition \ref{Prop: relative Kazdhan lemma} implies that each of these Hermitian spaces are split. In fact, we can identify the self-dual lattices directly. For the factors of $\rH_\ga$, we have the decomposition
 \[
 \Lam_1 = (\Lam_1\cap W_{1,1})\oplus (\Lam_1\cap W_{-1,1})
 \] 
 and
 \[
\Lam_{-1}= (\Lam_{-1}\cap W_{1,-1})\oplus (\Lam_{-1}\cap W_{-1,-1}).
\]are both self-dual lattice  giving rise to a hyperspecial subgroup $\rH_\ga(\calo_F)$. 
\begin{Lem}\label{Lem: into the reg locus}
Writing $x=\ga y= (y_1,-y_{-1})\in \calq_1(\calo_F) \times \calq_{-1}(\calo_F)$, we have
\[
-y_{-1}\in \calq_{-1}^{\heartsuit,1}(\calo_F)\:\text{ and }\:y_1\in \calq_{1}^{\heartsuit,-1}(\calo_F).
\]
\end{Lem}
\begin{proof}
The first assertion is immediate as $y_{-1}$ is topologically unipotent, so that $\overline{-y_{-1}}\notin D_1(k)$. Likewise, our analysis of eigenvalues of $y_{as}$ in Lemma \ref{Lem: eignvalue restrictions} tells us that no eigenvalue of $y_1$ is congruent to $-1$, so that $y_1\in \calq_{1}^{\heartsuit,-1}(\calo_F)$.
\end{proof}

The product decomposition of $(\G_\ga,\rH_\ga)$ induces a decomposition of the centralizer
\begin{equation}\label{eqn: stabilizer decomposition}
\rH_x = \rH_{x,1}\times \rH_{x,-1}\subset \U(W_{1}) \times \U(W_{-1})    
\end{equation}
and a resulting decomposition
 \[
 \ka_\ga=(\ka_{1},\ka_{-1}): \mathfrak{C}(\rH_{x,1},\U(W_1);F)\times\mathfrak{C}(\rH_{x,-1},\U(W_{-1});F)\lra\cc^\times.
 \]
We now prove a generalization on Proposition \ref{Prop: absolute descent} for the decomposition $x=\ga y$, the central point being that the uniqueness property of the topological Jordan decomposition holds implies a similar uniqueness principal for this decomposition.
 \begin{Lem}\label{eqn: almost there general}
 Let $x\in \calq_n(\calo_F)$ be regular semi-simple. With the notation as above,
   \begin{equation*}
 \Orb^\ka(x,\bfun_{\calq_n(\calo_F)})= \Orb^{\ka_1}(y_1,\bfun_{\calq_{1}(\calo_F)})\cdot \Orb^{\ka_{-1}}(- y_{-1},\bfun_{\calq_{-1}(\calo_F)}).
 \end{equation*}
 \end{Lem}
\begin{proof}
We have the two decompositions 
\[
x=x_{as}x_{tu} = \ga y.
\] Let $x'=x'_{as}x'_{tu}\in \calq_n(\calo_F)$ lie in the same stable orbit as $x$. Arguing as in Proposition \ref{Prop: absolute descent}, Proposition \ref{Prop: relative Kazdhan lemma} allows us to assume that $x_{as}=x_{as}'$. In particular, if we set $y'=y_{as}x_{tu}'$, we have a similar decomposition 
 \[
 x'=\ga y'=y'\ga,\text{ with }y'\in\G_\ga(\calo_F)
 \] such that $y$ and $y'$ lie in the same stable orbit. This implies that $x'\in \G_{\ga}(\calo_F)$. 
 
 Now, for any $h\in \rH(\Fbar)$ such that $x=hx'h^{-1},$ the uniqueness of the topological Jordan decomposition forces $x_{as}=hx_{as}h^{-1}$, so that $h\in \G_{x_{as}}(\Fbar)$. Applying this to the decomposition
 \[
 x_{as}=\ga y_{as} = (h\ga h^{-1})(h y'_{as}h^{-1}),
 \]
along with $h=(h_0,h_1,h_{-1})\in \G_{x_{as}}(\Fbar)$ preserving the decomposition of $x_{as}$, we conclude that
\[
\ga= h\ga h^{-1}.
\] 
In particular, $x$ and $x'$ lie in the same stable $\rH_\ga$-orbit. Noting that $y\in \calq_\ga^{rss}(F)$ by Lemma \ref{Lem: new decomp}, the natural generalization of Lemma \ref{Lem: all good} holds and the proof now follows mutatis mutandis as in Proposition \ref{Prop: absolute descent}.
\end{proof}

\subsection{Descent on the endoscopic side}Suppose now that $(x_a,x_b)\in \calq_{a,\al}(F)\times \calq_{b,\be}(F)$ matches $x$. Comparing characteristic polynomials, it follows from the definition of matching of orbits, Lemma \ref{Lem: characteristic comparison}, and the definition of strongly compact elements that $x_a\in U(V_a\oplus V_\al)$ and $x_b\in U(V_b\oplus V_\be)$ are strongly compact. In particular, there exist topological Jordan decompositions
\[
x_a=x_{a,as}x_{a,tu}\:\:\text{ and }\:\:x_b=x_{b,as}x_{b,tu}.
\]

Running the above argument in each case gives the descendants
\[
(\G_{\ga_a},\rH_{\ga_a})\: \text{ and }\: (\G_{\ga_b},\rH_{\ga_b})
\]
where we have  $(x_a,x_b) = (\ga_ay_a,\ga_by_b),$ with the obvious meaning for the notation. Write 
\[
W_a = V_a\oplus V_\al\:\:\text{ and }\:\:W_b=V_b\oplus V_\be.
\]
The action of $\ga_a$ on $W_a$ and $\ga_b$ on $W_b$ induce analogous decompositions
\[
W_a=W_{a,1}\oplus W_{a,-1}\:\text{  and  }\:W_b=W_{b,1}\oplus W_{b,-1}
\]
 and the centralizers are of the form
\[
\G_{\ga_a} = \U(W_{a,1})\times \U(W_{a,-1})\:\text{ and }\:\G_{\ga_b} = \U(W_{b,1})\times \U(W_{b,-1}),
\]
and the groups $\rH_{\ga_a}$ and $\rH_{\ga_b}$ are the appropriate products of unitary subgroups as above.
\begin{Lem}\label{eqn: stable almost there general}
  If $(\al,\be)\neq (I_a,I_b),$ then 
  \[
  \Orb^\ka(x,\bfun_{\calq_n(\calo_F)})=0.
  \]
  Otherwise, the stable orbital integral
\[
\SO((x_a,x_b),\bfun_{\calq_a(\calo_F)}\otimes \bfun_{\calq_b(\calo_F)})
\]
equals
 \begin{equation*}
 \SO\left((y_{a,1},y_{b,1}),\bfun_{\calq_{a,1}(\calo_F)}\otimes \bfun_{\calq_{b,1}(\calo_F)}\right)\cdot\SO\left((-y_{a,-1},-y_{b,-1}),\bfun_{\calq_{a,-1}(\calo_F)}\otimes \bfun_{\calq_{b,-1}(\calo_F)}\right),
 \end{equation*}
 where $\calq_{a,\pm1}$ and $\calq_{b,\pm1}$ are the descendants associated to $(\ga_a,\ga_b)\in \calq_a(F)\times\calq_b(F)$.
\end{Lem}
\begin{proof}
 If $(\al,\be)\neq (I_a,I_b),$ then at least one of the forms on the four Hermitian spaces 
\[
W_{a,1},\: W_{a,-1}\quad\text{ or }\quad  W_{b,1},\: W_{b,-1}
\]is not split. Combining Lemma \ref{eqn: almost there general} with the vanishing statement in Proposition \ref{Prop: very reg FL} now implies that
\[
\Orb^\ka(x,\bfun_{\calq_n(\calo_F)})=0.
\]
Thus, we assume may that $(\al,\be)=(I_a,I_b)$, and the result follows from Proposition \ref{Prop: absolute descent}, augmented as in Lemma \ref{eqn: almost there general}.
\end{proof}

\subsection{Descent of transfer factors}\label{Section: descent3} Finally, we consider the transfer factor.
\begin{Lem}\label{eqn: almost there transfer factor}
Suppose that $x\in\calq_n(\calo_F)$ and $(x_a,x_b)\in\calq_a(F)\times \calq_b(F)$ match, with $x$ satisfying the assumption of Proposition \ref{Prop: final cases}. Then
\begin{equation*}
    \De_{rel}((x_a,x_b),x) = \De_{rel}((y_{a,1},y_{b,1}),y_1)\cdot\De_{rel}((-y_{a,-1},-y_{b,-1}),-y_{-1}).
\end{equation*}
\end{Lem}
\begin{proof}The matching induces a partition of multi-sets
\[
R=R_a\bigsqcup R_b,
\] where $R\subset \Fbar^\times$ are the roots of $\chi_x$ (with multiplicity) and $R_a$ (resp., $R_b$) are the  roots of the invariant of $x_a$ (resp., $x_b$). On the other hand, the decomposition
\[
W= W_{1}\oplus W_{-1}
\]
induces a partition of multi-sets
\[
R= R_1\bigsqcup R_{-1}.
\]
The actions of $\ga_{a}$ on $W_a$ and $\ga_b$ on $W_b$ induce analogous decompositions
\[
R_a=R_{a,1}\bigsqcup R_{a,-1}\:\:\text{  and  }\:\:R_b= R_{b,1}\bigsqcup R_{b,-1}.
\]
It is immediate from the definition of matching that these partitions are compatible. 

Recalling the formulas for the transfer factor in Appendix \ref{Section: transfer factor}, consider the relative discriminant
\[
D_{a,b}(x) = \prod_{(z_a,z_b)\in R_a\times R_b}(z_a-z_b).
\]
If $z_a\in R_{a,1}$ and $z_b\in R_{b,-1},$ then Lemma \ref{Lem: eignvalue restrictions} implies that $z_a-z_b$ is a unit in $\calo_{\Fbar}$, and similarly if $z_a\in R_{a,-1}$ and $z_b\in R_{b,1}$. 
Thus, recalling $x= (y_1,-y_{-1})\in \calq_{\ga}(\calo_F)$, if we set
\[
D_{(a,\nu),(b,\nu)}(\nu y_\nu) = \prod_{(z_a,z_b)\in R_{a,\nu}\times R_{b,\nu}}(z_a-z_b)
\]we compute that
\[
|D_{a,b}(x)|_F = \left|D_{(a,1),(b,1)}(y_1)\right|_F\cdot\left|D_{(a,-1),(b,-1)}(-y_{-1})\right|_F
\]
and 
\[
\eta_{E/F}(D_{a,b}(x)) = \eta_{E/F}\left(D_{(a,1),(b,1)}(y_1)\right)\cdot\eta_{E/F}\left(D_{(a,-1),(b,-1)}(-y_{-1})\right).
\]
In particular, if $(x,(x_a,x_b))$ is a good matching pair, then the definitions of the transfer factors in (\ref{eqn: relative transfer factors}) and Section \ref{Section: transfer factor} immediately imply the lemma.

When $(x,(x_a,x_b))$ is not a good matching pair, the assumption that $(\al,\be)= (I_a,I_b)$ implies that we can fix an embedding
\[
\varphi_{a,b}:\Herm(V_a)\oplus \Herm(V_b)\hra \Herm(V_n),
\]
and set $A':=\varphi_{a,b}(R_a(x_a),R_b(x_b))$. Then $A'$ is stably conjugate to $R(x)\in \Herm(V_n)$, so Lemma \ref{Lem: kernel isom} tells us that there exists $x'\in \calq(F)$ in the same stable orbit as $x$ such that $R(x')=A'$. In particular, $x'$ and $(x_a,x_b)$ are a good matching pair. 

Letting $x'=\ga'y'$ be the analogous decomposition, we similarly write $x'=(y'_1,-y'_{-1})\in \calq_{\ga'}(F)$.  Suppose that $x'=hxh^{-1}$ for $h\in H(\Fbar)$. As we have seen in the proof of Lemma \ref{eqn: almost there general}, uniqueness of the topological Jordan decomposition forces 
\[
\ga'=h\ga h^{-1}\:\:\text{ and }\:\:y_{\pm1}'=hy_{\pm1}h^{-1}.
\]
 This proves that under the isomorphism induced by the decomposition (\ref{eqn: stabilizer decomposition}) we have
\begin{align*}
    H^1(F,H_x)&\cong H^1(F,\rH_{x,1})\times H^1(F,\rH_{x,-1}),\\
    \inv(x,x')&\longmapsto(\inv(y_1,y_1'),\inv(-y_{-1},-y_{-1}')).
\end{align*}
In particular,
\[
\ka( \inv(x,x')) = \ka_1(\inv(y_1,y_1'))\cdot\ka_{-1}(\inv(-y_{-1},-y_{-1}'));
\]
The formulas for the transfer factor (\ref{eqn: transfer not nice}) now implies the lemma for all matching pairs $(x,(x_a,x_b))$.
\end{proof}
\subsection{Final cases of Theorem \ref{Thm: fundamental lemma}}
We may now complete the proof of Theorem \ref{Thm: fundamental lemma}.
\begin{Prop}\label{Prop: final cases}
Suppose $x\in\calq_n^{rss}(\calo_F)$ such that 
\[
\overline{x}\in (D_1\cap D_{-1})(k).
\]
and that $(x_a,x_b)\in \calq_{a,\al}(F)\times\calq_{b,\be}(F)$ match. Then the fundamental lemma holds for $(x, (x_a,x_b))$. That is, if $\ka$ is the character associated to the endoscopic datum, we have
\[
\De_{rel}((x_a,x_b),x)\Orb^\ka(x,\bfun_{\calq_n(\calo_F)})=\begin{cases}\SO((x_a,x_b),\bfun_{\calq_a(\calo_F)}\otimes \bfun_{\calq_b(\calo_F)})&:(\al,\be)= (I_a,I_b),\\\qquad\qquad\qquad 0&:(\al,\be)\neq (I_a,I_b).\end{cases}
\]
\end{Prop}
\begin{proof}
 Lemma \ref{Lem: into the reg locus} allows us to apply Proposition \ref{Prop: very reg FL} to the descended orbital integrals and transfer factors appearing in Lemmas \ref{eqn: almost there general}, \ref{eqn: stable almost there general}, and \ref{eqn: almost there transfer factor} on the descendant $\calq_1 \times \calq_{-1}$. The resulting matching of orbital integrals gives the identity.
\end{proof}

\begin{appendix}

\section{Endoscopy for unitary Lie algebras}\label{Section: endoscopy roundup}
We recall the necessary facts from the theory of endoscopy for unitary Lie algebras. We follow \cite{Leslieendoscopy} closely. Assume $E/F$ is a quadratic extension of $p$-adic fields and let $W$ be a $n$-dimensional Hermitian space over $E$. As previously noted, we will work with the twisted Lie algebra
\[
\Herm(W)=\{x\in \End(W): \la xu,v\ra=\la u, xv\ra\}.
\]

Let $\de\in \Herm(W)$ be regular and semi-simple. 
Recalling that the set of rational conjugacy classes $\calo_{st}(\de)$ in the stable conjugacy class of $y$ form a $\cald(T_\de,\U(W);F)$-torsor, we have a map
\begin{equation*}
\inv(\de,-):\calo_{st}(\de)\iso \cald(T_\de,\U(W);F)
\end{equation*}trivializing the torsor by fixing the orbit $[\de]$ as the base point. This map is given by 
\[
[\de']\mapsto \inv(\de,\de'):=[\sig\in \Gal(\Fbar/F)\mapsto \sig(g)^{-1}g], 
\]
where $g\in \GL(W)$ such that $\de'=\Ad(g)(\de)$.

\subsection{Endoscopy for unitary Lie algebras}\label{Section: endoscopy defs}

An elliptic endoscopic datum for $\Herm(W)$ is the same as a datum for the group $U(W),$ namely a triple $$(\U(V_a)\times \U(V_b),s,\eta),$$  where $a+b=n$. Here $s\in \hat{U}(W)$ a semi-simple element of the Langlands dual group of $U(W)$, and an embedding $$\eta:\hat{U}(V_a)\times\hat{U}(V_b)\hra \hat{U}(W)$$ identifying $\hat{U}(V_a)\times\hat{U}(V_b)$ with the neutral component of the centralizer of $s$ in the $L$-group ${}^LU(W)$. Fixing such a datum, we consider the endoscopic Lie algebra $\Herm(V_a)\oplus \Herm(V_b)$. Let $\de\in \Herm(W)$ and $(\de_a,\de_b)\in\Herm(V_a)\oplus \Herm(V_b)$ be regular semi-simple. 

Denote $W_{a,b}=V_a\oplus V_b$. In the non-archimedean case, the isomorphism class of $W_{a,b}$ is uniquely determined by those of $V_a$ and $V_b$ \cite[Theorem 3.1.1]{jacobowitz1962hermitian}. 

\subsection{Matching of orbits}\label{Section: endo matching} We first recall the notion of Jacquet--Langlands transfer between two non-isomorphic Hermitian spaces $W$ and $W'$. If we identify the underlying vector spaces (but not necessarily the Hermitian structures)
\begin{equation*}
  W\cong E^n\cong W',  
\end{equation*}
we have embeddings 
\[
\Herm(W),\:\Herm(W')\hra\fgl_n(E).
\]
Then $\de\in \Herm(W)$ and $\de'\in \Herm(W')$ are said to be \textbf{Jacquet--Langlands transfers} if they are $\GL_n(E)$-conjugate in $\fgl_n(E)$.  This is well defined since the above embeddings are determined up to $\GL_n(E)$-conjugacy. Note that if $\de$ and $\de'$ are Jacquet--Langlands transfers, then
\[
\de'=\Ad(g)(\de)
\]
for some $g\in \GL(W)$ and we obtain a well-defined cohomology class
\[
\inv(\de,\de')=[\sig\in \Gal(\Fbar/F)\mapsto g^{-1}\sig(g)]\in H^1(F,T_\de)
\]
extending the invariant map on $\cald(T_\de;\U(W);F)$.

\begin{Def}\label{Def: endoscopic matching}
In the case that $W'=W_{a,b}$, we have an embedding
 $$\phi_{a,b}:\Herm(V_a)\oplus\Herm(V_b)\hra \Herm(W_{a,b}),$$ well defined up to conjugation by $U(W_{a,b})$. We say that $\de$ and $(\de_a,\de_b)$ are \textbf{transfers (or are said to match)} if $\de$ and $\phi_{a,b}(\de_a,\de_b)$ are Jacquet--Langlands transfers in the above sense.
\end{Def}

For later purposes, if $W\cong W_{a,b}$, we say that a matching pair $y$ and $(\de_a,\de_b)$ are a \textbf{\emph{nice matching pair}} if we may choose $\phi_{a,b}$ so that
 \[
 \phi_{a,b}(\de_a,\de_b) = \de.
 \]
We remark that being Jacquet--Langlands transfers is equivalent to belonging to the same stable conjugacy class, while the notation of a nice matching pair is equivalent  $\phi_{a,b}(\de_a,\de_b)$ and $\de$ lying in the same rational conjugacy class.

\subsubsection{Orbital integrals}\label{Section:endoscopy character}
For $\de\in \Herm(W)^{rss}$ and $f\in C_c^\infty(\Herm(W))$, we define the orbital integral
\[
\Orb(\de,f)=\int_{T_\de\backslash U(W)}f(g^{-1}\de g)d\dot{g},
\]
where $dg$ is a Haar measure on $U(W)$, $dt$ is the unique normalized Haar measure on the torus $T_\de$, and $d\dot{g}$ is the invariant measure such that $dt d\dot{g}=dg$.

To an elliptic endoscopic datum $(\U(V_a)\times \U(V_b),s,\eta)$ and regular semi-simple element $\de\in \Herm(W)$, there is a natural character (see Section \ref{Section: Prelim inv} for notation)
\[
\ka:\mathfrak{C}(T_\de,\U(W);F)\cong\cald(T_\de,\U(W);F)\lra \cc^\times,
\]
which may be computed as follows. For matching elements $\de$ and $(\de_a,\de_b)$,
\begin{equation*}
H^1(F,T_\de)=\prod_{S_1}\zz/2\zz=\prod_{S_1(a)}\zz/2\zz\times\prod_{S_1(b)}\zz/2\zz=H^1(F,T_{\de_a}\times T_{\de_b}),
\end{equation*}
where the notation indicates which elements of $S_1$ arise from the torus $T_{\de_a}$ or $T_{\de_b}$.
\begin{Lem}\cite[Proposition 3.10]{Xiao}\label{Lem: endo character}
 Consider the character $\tilde{\ka}: H^1(F,T_\de)\to \cc^\times$ such that on each $\zz/2\zz$-factor arising from $S_1(a)$, $\tilde{\ka}$ is the trivial map, while it is the unique non-trivial map on each $\zz/2\zz$-factor arising from $S_1(b)$. Then $$\ka=\tilde{\ka}|_{\mathfrak{C}(T_\de,\U(W);F)}.$$ 
\end{Lem}

Using the invariant map
\begin{equation*}
\inv(\de,-):\calo_{st}(\de)\iso \cald(T_\de,\U(W);F),
\end{equation*}we form the $\ka$-orbital integral of $f\in C_c^\infty(\Herm(W_1))$
\[
\Orb^\ka(\de,f) = \sum_{[\de']}\ka(\inv(\de,\de'))\Orb(\de',f).
\]
When $\ka=1$ is trivial, write $\SO=\Orb^\ka$.
\subsubsection{Transfer factors}\label{Section: transfer factor}
We now recall the \emph{transfer factor} of Langlands-Sheldstad and Kottwitz. This is a function
\[
\De:[\Herm(V_a)\oplus \Herm(V_b)]^{rss}\times  \Herm(W)^{rss}\to \cc.
\]
The two important properties are
\begin{enumerate}
    \item $\De((\de_a,\de_b),\de) = 0$ if $\de$ does not match $(\de_a,\de_b),$
    \item\label{second} if $\de$ is stably conjugate to $\de'$, then 
    \[
    \De((\de_a,\de_b),\de)\Orb^\ka(\de,f) = \De((\de_a,\de_b),\de')\Orb^\ka(\de',f).
    \]
\end{enumerate}
While the general definition, given in \cite{LanglandsShelstad} for the group case and \cite{kottwitztransfer} in the quasi-split Lie algebra setting, is subtle, our present setting enjoys the following simplified formulation when $E/F$ is unramified.  



When $\de\in \Herm(W)$ and $(\de_a,\de_b)\in \Herm(V_a)\oplus \Herm(V_b)$ do not match, we set
\[
\De((\de_a,\de_b),\de)=0.
\]
Now suppose that $\de$ and $(\de_a,\de_b)$ match. We define the relative discriminant
\[
D_{a,b}(\de)=\prod_{x_a,x_b}(x_a-x_b),
\]
where $x_a$ (resp. $x_b$) ranges over the eigenvalues of $\de_a$ (resp. $\de_b$) in $\Fbar$. 

Recall our notation $W_{a,b}=V_a\oplus V_b$ and first assume that $W\cong W_{a,b}$ and that $\de$ and $(\de_a,\de_b)$ are a nice matching pair. In this case, the transfer factor is then given by
\begin{equation*}
\De((\de_a,\de_b),\de):=\eta_{E/F}(D_{a,b}(\de))|D_{a,b}(\de)|_F,
\end{equation*}
where $\eta_{E/F}$ is the quadratic character associated to $E/F$.

Now for any matching pair $\de$ and $(\de_a,\de_b)$, let
\[
\de'=\phi_{a,b}(\de_a,\de_b)\in \Herm(W_{a,b}).
\]
As discussed in Section \ref{Section: endo matching}, $\de$ and $\de'$ are Jacquet--Langlands transfers of each other and we set
\begin{align}\label{eqn: transfer not nice}
   \De((\de_a,\de_b),\de) = \kappa(\inv(\de,\de'))\eta_{E/F}(D_{a,b}(\de))|D_{a,b}(\de)|_F=\kappa(\inv(\de,\de'))\De((\de_a,\de_b),\de'),
\end{align}

where $\ka:H^1(F,T_\de)\to \cc^\times$ is the character arising from the datum $(\U(V_a)\times \U(V_b),s,\eta)$ and $\inv$ is the extension of the invariant map discussed in Section \ref{Section: endo matching}.  This property ensures that $\De$ satisfies (\ref{second}) above.

We use these explicit formulas in Sections \ref{Section: cayley}, \ref{Section: very regular}, and \ref{Section: descent3}.

\end{appendix}



\newpage
\bibliographystyle{alpha}

\bibliography{bibs}
\end{document}